\documentclass[reqno]{amsart}

\usepackage{amsmath}
\usepackage{amssymb}
\usepackage{amsthm}
\usepackage[hyperfootnotes=false]{hyperref}
\usepackage{mathrsfs}

\usepackage[backend=bibtex,style=alphabetic,firstinits,url=false,doi=false]{biblatex}

\numberwithin{equation}{section}
\newtheorem{theorem}[equation]{Theorem}

\newtheorem{proposition}[equation]{Proposition}
\newtheorem{corollary}[equation]{Corollary}
\newtheorem{lemma}[equation]{Lemma}

\theoremstyle{definition}
\newtheorem{example}[equation]{Example}
\newtheorem{definition}[equation]{Definition}

\newcommand{\define}[1]{\textit{#1}}
\newcommand{\nbar}{\vert\!\vert}
\newcommand{\gnbar}{\vert\!\vert\!\vert}
\newcommand{\bnbar}{\Bigg\vert\!\Bigg\vert}

\newcommand{\ultra}[1]{\mathsf{{#1}}}
\renewcommand{\epsilon}{\varepsilon}
\renewcommand{\emptyset}{\varnothing}
\newcommand{\intd}{\,\mathrm{d}}
\newcommand{\condex}[2]{\mathbb{E}({#1}\vert{#2})}

\newcommand{\haar}{\mathrm{m}}
\newcommand{\divides}{\vert}
\newcommand{\diagonal}{\triangle}
\newcommand{\prideal}[1]{\mathfrak{{#1}}}
\newcommand{\ringint}{\mathcal{O}}

\DeclareMathOperator{\aip}{AIP}
\DeclareMathOperator{\C}{C}

\DeclareMathOperator*{\dlim}{D-lim}
\DeclareMathOperator*{\clim}{C-lim}

\DeclareMathOperator{\conv}{\ast}

\DeclareMathOperator{\fs}{FS}

\DeclareMathOperator{\lp}{L}

\DeclareMathOperator{\ip}{IP}

\DeclareMathOperator{\symdiff}{\triangle}

\newcommand{\upperdens}{\mathrm{d}^*}

\title[Polynomial multiple recurrence over rings of integers]{Polynomial multiple recurrence over rings of integers}
\author{Vitaly Bergelson}
\thanks{The first author gratefully acknowledges the support of the NSF under grant DMS-1162073.}
\author{Donald Robertson}
\date{\today{}}

\begin{document}

\begin{abstract}
We generalize the polynomial Szemer\'{e}di theorem to intersective polynomials over the ring of integers of an algebraic number field, by which we mean polynomials having a common root modulo every ideal.
This leads to the existence of new polynomial configurations in positive-density subsets of $\mathbb{Z}^m$ and strengthens and extends recent results of Bergelson, Leibman and Lesigne on polynomials over the integers.
\end{abstract}

\maketitle

\section{Introduction}

Let $T$ be a measure-preserving action of $\mathbb{Z}$ on a probability space $(X,\mathscr{B},\mu)$ and fix $B$ in $\mathscr{B}$ with $\mu(B) > 0$.
Furstenberg's ergodic Szemer\'{e}di theorem \cite{MR0498471} implies that the set
\begin{equation*}
\{ n \in \mathbb{Z} : \mu(B \cap T^n B \cap \cdots \cap T^{kn}B) > 0 \}
\end{equation*}
is syndetic, which means that finitely many of its shifts cover $\mathbb{Z}$.
The polynomial ergodic Szemer\'{e}di theorem in \cite{MR1325795} implies, in particular, that
\begin{equation}
\label{eqn:polyReturnSet}
R = \{ n \in \mathbb{Z} : \mu(B \cap T^{p_1(n)} B \cap \cdots \cap T^{p_k(n)}B) > 0 \}
\end{equation}
has positive lower density, meaning that
\[
\liminf_{N \to \infty} \frac{|R \cap \{ 1,\dots,N\}|}{N} > 0,
\]
for any $p_1,\dots,p_k \in \mathbb{Z}[x]$ each having zero constant term.
It was shown in \cite{MR1411223} that \eqref{eqn:polyReturnSet} is syndetic under the same assumptions, and the later work \cite{MR1692634} implies it is large in the stronger sense (defined below) of being $\ip^*$.

The task of determining precisely which families $p_1,\dots,p_k$ of polynomials have the property that \eqref{eqn:polyReturnSet} is syndetic was undertaken in \cite{MR2435427}.
There it was shown polynomials $p_1,\dots,p_k$ have the property that \eqref{eqn:polyReturnSet} is syndetic whenever $T$ is an action of $\mathbb{Z}$ on $(X,\mathscr{B},\mu)$ and $\mu(B) > 0$ if and only if the polynomials are \define{jointly intersective}, which means that for any finite index subgroup $\Lambda$ of $\mathbb{Z}$, one can find $\zeta$ in $\mathbb{Z}$ such that $\{ p_1(\zeta),\dots,p_k(\zeta) \} \subset \Lambda$.

The polynomial ergodic Szemer\'{e}di theorem in \cite{MR1325795} actually implies the following multi-dimensional result: for any action $T$ of $\mathbb{Z}^m$ on a probability space $(X,\mathscr{B},\mu)$ and any $B$ with $\mu(B) > 0$ the set
\begin{equation}
\label{eqn:polyMultiReturns}
\{ n \in \mathbb{Z}^d : \mu(B \cap T^{p_1(n)} B \cap \cdots \cap T^{p_k(n)} B) > 0 \}
\end{equation}
has positive lower density for any polynomial mappings $p_1,\dots,p_k : \mathbb{Z}^d \to \mathbb{Z}^m$ each having zero constant term.
In \eqref{eqn:polyMultiReturns} and below we write $T^{p_i(u)}$ for $T_1^{p_{i,1}(u)} \cdots T_m^{p_{i,m}(u)}$ when $p_i = (p_{i,1},\dots,p_{i,m})$.
As in the $m = 1$ case above, \cite{MR1692634} implies that \eqref{eqn:polyMultiReturns} is $\ip^*$.
There is no known characterization of those polynomial mappings $p_1,\dots,p_k$ for which \eqref{eqn:polyMultiReturns} is non-empty.
By considering finite systems, one can show that joint intersectivity (defined below in general) is a necessary condition; it is conjectured in \cite{MR2435427} that it is also sufficient.

Since \cite{MR0498471}, the sizes of sets such as \eqref{eqn:polyReturnSet} have been studied by considering the limiting behavior of averages such as
\begin{equation}
\label{eqn:recAverage}
\frac{1}{|\Phi_N|} \sum_{u \in \Phi_N} \mu(B \cap T^{p_1(u)} B \cap \cdots \cap T^{p_k(u)} B)
\end{equation}
where $N \mapsto \Phi_N$ is some sequence of longer and longer intervals in $\mathbb{Z}$.
In \cite{MR2435427} the works of Host and Kra~\cite{MR2150389} and Ziegler~\cite{MR2257397} on characteristic factors are combined with \cite{MR2151605} to prove that the limiting behavior of the average \eqref{eqn:recAverage} can be approximated arbitrarily well by replacing $(X,\mathscr{B},\mu)$ with quotients $G/\Gamma$ of certain nilpotent Lie groups by a cocompact subgroup on which $\mathbb{Z}$ acts via $T(g\Gamma) = ag\Gamma$ for some $a \in G$.
Upon passing to this more tractable setting, it is shown in \cite{MR2435427} that \eqref{eqn:recAverage} is positive in the limit as $N \to \infty$ when $p_1,\dots,p_k$ are jointly intersective.

It is not possible to proceed like this when studying \eqref{eqn:polyMultiReturns} because there is currently no general version of the work of Host and Kra~\cite{MR2150389} and Ziegler~\cite{MR2257397} for actions of $\mathbb{Z}^m$.
In this paper we enlarge the class of polynomial mappings $p_1,\dots,p_k$ for which \eqref{eqn:polyMultiReturns} is known to be non-empty by working with polynomials over rings of integers of algebraic number fields.
As we will see, this is a setting where it is possible to reduce to the case of commuting translations on homogeneous spaces of nilpotent Lie groups, which will allow us to show that \eqref{eqn:polyMultiReturns} is large.
Our techniques also allow us to improve upon the main result in \cite{MR2435427} by strengthening the largeness property of the set \eqref{eqn:polyReturnSet}.
To describe our results we recall some definitions.

\begin{definition}
Let $R$ be a commutative ring with identity.
Polynomials $p_1,\dots,p_k$ in $R[x_1,\dots,x_d]$ are said to be \define{jointly intersective} if, for any finite index subgroup $\Lambda$ of $R$, one can find $\zeta$ in $R^d$ such that $\{ p_1(\zeta),\dots,p_k(\zeta) \} \subset \Lambda$.
When $d = 1$ we say that $p_1$ is \define{intersective}.
\end{definition}

See Section~\ref{sec:singlePolyRec} for some discussion of intersective polynomials.
We also need the following notions of size.

\begin{definition}
Let $G$ be an abelian group.
An \define{IP set} in $G$ is any subset of $G$ containing a set of the form
\begin{equation*}
\fs(x_n) := \bigg\{ \sum_{n \in \alpha} x_n : \emptyset \ne \alpha \subset \mathbb{N}, |\alpha| < \infty \bigg\}
\end{equation*}
for some sequence $x_n$ in $G$.
A subset of $G$ is \define{$\ip^*$} if its intersection with every $\ip$ set in $G$ is non-empty, and \define{$\ip^*_+$} if it is a shift of an $\ip^*$ set.
The term $\ip$ was introduced in \cite{MR531271}, the initials standing for ``idempotence'' or ``infinite-dimensional parallelopiped''.
The \define{upper Banach density} of a subset $S$ of $G$ is defined by
\begin{equation*}
\upperdens(S) = \sup \left\{ \upperdens_\Phi(S) : \Phi \textup{ a F\o{}lner sequence in } G \right\}
\end{equation*}
where
\begin{equation*}
\upperdens_\Phi(S) = \limsup_{N \to \infty} \frac{|S \cap \Phi_N|}{|\Phi_N|}
\end{equation*}
and a \define{F\o{}lner sequence} in $G$ is a sequence $N \mapsto \Phi_N$ of finite, non-empty subsets of $G$ such that
\begin{equation*}
\lim_{N \to \infty} \frac{|(g + \Phi_N) \cap \Phi_N|}{|\Phi_N|} = 1
\end{equation*}
for all $g$ in $G$.
Lastly, $S \subset G$ is \define{$\aip^*$} (with A standing for ``almost'') if it is of the form $A \backslash B$ where $A$ is an $\ip^*$ subset of $G$ and $\upperdens(B) = 0$, and $S$ is \define{$\aip^*_+$} if it is a shift of an $\aip^*$ set.
\end{definition}

We can now state our main result.
Given an algebraic number field $L$, write $\ringint_L$ for its ring of integers.

\begin{theorem}
\label{thm:mainTheorem}
Let $L$ be an algebraic number field and let $p_1,\dots,p_k$ be jointly intersective polynomials in $\ringint_L[x_1,\dots,x_d]$.
For any ergodic action $T$ of the additive group of $\ringint_L$ on a compact metric probability space $(X,\mathscr{B},\mu)$ and any $B \in \mathscr{B}$ with $\mu(B) > 0$ there is $c > 0$ such that
\begin{equation}
\label{eqn:mainTheorem}
\{ u \in \ringint_L^d : \mu(B \cap T^{p_1(u)} B \cap \cdots \cap T^{p_k(u)} B) \ge c \}
\end{equation}
is $\aip^*_+$.
\end{theorem}

In particular, taking $L = \mathbb{Q}$ shows that \eqref{eqn:polyReturnSet} is an $\aip^*_+$ subset of $\mathbb{Z}$.
We will see in Example~\ref{ex:aipNotSyndetic} that being $\aip^*_+$ is a stronger property than being syndetic, so Theorem~\ref{thm:mainTheorem} constitutes a strengthening of \cite[Theorem~1.1]{MR1325795}.

The following version of the Furstenberg correspondence principle allows us to use Theorem~\ref{thm:mainTheorem} to find polynomial configurations in large subsets of $\ringint_L$.

\begin{theorem}
\label{thm:fcp}
For any $E \subset \ringint_L$ there is an ergodic action $T$ of $\ringint_L$ on a compact metric probability space $(X,\mathscr{B},\mu)$ and $B \in \mathscr{B}$ with $\mu(B) = \upperdens(E)$ such that
\begin{equation}
\label{eqn:fcp}
\upperdens \big( (E-u_1) \cap \cdots \cap (E - u_k) \big) \ge \mu( T^{u_1} B \cap \cdots \cap T^{u_k} B)
\end{equation}
for every $u_1,\dots,u_k$ in $\ringint_L$.
\end{theorem}

That one can associate an ergodic action with $E$ was first proved in \cite{MR2138068} using ideas from \cite{MR603625}, and the correspondence principle stated above can be proved exactly as in \cite{MR2138068}.
Combining Theorems~\ref{thm:mainTheorem} and \ref{thm:fcp} gives the following combinatorial result.

\begin{theorem}
\label{thm:intComb}
Let $L$ be an algebraic number field and let $E \subset \ringint_L$ have positive upper Banach density.
For any jointly intersective polynomials $p_1,\dots,p_k$ in $\ringint_L[x_1,\dots,x_d]$ there is a constant $c > 0$ such that the set
\begin{equation}
\label{eqn:intComb}
\{ u \in \ringint_L^d : \upperdens \big(E \cap (E - p_1(u)) \cap \cdots \cap (E - p_k(u)) \big) \ge c \}
\end{equation}
is $\aip^*_+$.
\end{theorem}

Whenever $\ringint_L$ is finitely partitioned, one of the partitions has positive upper Banach density.
As a result, Theorem~\ref{thm:intComb} yields new examples of the polynomial van der Waerden theorem, extending \cite[Theorem~1.5]{MR2435427}.

\begin{corollary}
Let $L$ be an algebraic number field.
For any finite partition $E_1 \cup \cdots \cup E_k$ of $\ringint_L$ there is $1 \le i \le k$ such that, for any jointly intersective polynomials $p_1,\dots,p_k \in \ringint_L[x_1,\dots,x_d]$ the set \eqref{eqn:intComb} is $\aip^*_+$.
\end{corollary}

So far, such polynomial van der Waerden results have only been proved via multiple recurrence of measure-preserving dynamical systems.
It would be interesting to have a proof that only used topological dynamics, or a purely combinatorial proof.

Upon fixing a basis $e_1,\dots,e_m$ for $\ringint_L$ as a $\mathbb{Z}$ module, defining actions $T_1,\dots,T_m$ of $\mathbb{Z}$ by $T_i^n = T^{ne_i}$, and writing
\begin{equation}
\label{eqn:polyDecomp}
p_i(u) = p_{i,1}(u)e_1 + \cdots + p_{i,m}(u) e_m
\end{equation}
for some polynomials $p_{i,j}$ in $\mathbb{Z}[x_1,\dots,x_{dm}]$, we see that Theorem~\ref{thm:mainTheorem} implies
\begin{equation*}
\left\{ u \in \mathbb{Z}^{md} : \int 1_B \prod_{i=1}^k T_1^{p_{i,1}(u)} \cdots T_m^{p_{i,m}(u)} 1_B \intd\mu > 0\right\}
\end{equation*}
is $\aip^*_+$, extending \cite[Theorem~A]{MR1325795} to certain families of intersective polynomials.
Indeed, if for some polynomials $p_{i,j}$ from $\mathbb{Z}[x_1,\dots,x_d]$, one can find an algebraic number field $L$, jointly intersective polynomials $p_1,\dots,p_k$ in $\ringint_L[x_1,\dots,x_d]$, and a basis $e_1,\dots,e_m$ for $\ringint_L$ over $\mathbb{Z}$ such that \eqref{eqn:polyDecomp} holds, then the polynomial mappings $(p_{1,1},\dots,p_{1,m}),\dots,(p_{k,1},\dots,p_{k,m}) : \mathbb{Z}^d \to \mathbb{Z}^m$ are good for recurrence.

It would be interesting to know whether \eqref{eqn:mainTheorem} is $\aip^*_+$ without the ergodicity assumption.
We show that it is syndetic.

\begin{theorem}
\label{thm:mainTheoremNonerg}
Let $L$ be an algebraic number field and let $p_1,\dots,p_k$ be jointly intersective polynomials in $\ringint_L[x_1,\dots,x_d]$.
For any action $T$ of the additive group of $\ringint_L$ on a compact metric probability space $(X,\mathscr{B},\mu)$ and any $B \in \mathscr{B}$ with $\mu(B) > 0$ there is $c > 0$ such that
\begin{equation}
\label{eqn:mainTheoremNoErgodic}
\{ u \in \ringint_L^d : \mu(B \cap T^{p_1(u)} B \cap \cdots \cap T^{p_k(u)} B) \ge c \}
\end{equation}
is syndetic.
\end{theorem}

Our proof of Theorem~\ref{thm:mainTheorem} consists of two main steps.
First we show, by combining Leibman's polynomial convergence result \cite{MR2151605} with Griesmer's description \cite{griesmerThesis} of characteristic factors for certain actions of $\mathbb{Z}^m$, that upon restricting our attention to a very large subset of $\ringint_L^d$ -- one whose complement has zero upper Banach density -- it suffices to consider \eqref{eqn:mainTheorem} when $(X,\mathscr{B},\mu)$ has the structure of a nilrotation, the definition of which we now recall.

\begin{definition}
By a \define{nilmanifold} we mean a homogeneous space $G/\Gamma$ where $G$ is a nilpotent Lie group and $\Gamma$ is a discrete, cocompact subgroup of $G$.
A \define{nilrotation} is an action $T$ of $\mathbb{Z}^m$ on a nilmanifold $G/\Gamma$ of the form $T^u(g\Gamma) = \phi(u)g\Gamma$ for some homomorphism $\phi : \mathbb{Z}^m \to G$.
The \define{nilpotency degree} of a nilrotation is the minimal length of a shortest central series for $G$.
\end{definition}

The second step in the proof of Theorem~\ref{thm:mainTheorem} is to use results from \cite{MR2435427} about polynomial orbits of nilrotations to show that, within the very large subset of $\ringint_L$ mentioned above, we can achieve the desired multiple recurrence.

It is natural to ask how large the intersection in \eqref{eqn:mainTheorem} can be.
When $k = 1$ we show it is as large as can be expected, extending results in \cite{MR628658}, \cite{MR487031} and \cite{MR516154}.

\begin{theorem}
\label{thm:singlePolyRec}
Let $L$ be an algebraic number field and let $p \in \ringint_L[x_1,\dots,x_d]$ be an intersective polynomial.
For any action $T$ of the additive group of $\ringint_L$ on a probability space $(X,\mathscr{B},\mu)$ and any $B$ in $\mathscr{B}$ the set
\begin{equation}
\label{eqn:singlePolyRec}
\{ u \in \ringint_L^d : \mu(B \cap T^{p(u)} B) > \mu(B)^2 - \epsilon \}
\end{equation}
is $\aip^*_+$ for any $\epsilon > 0$.
\end{theorem}

When $p$ has zero constant term one can use \cite[Theorem~1.8]{MR1417769} to show that \eqref{eqn:singlePolyRec} is $\ip^*$.
It follows immediately that \eqref{eqn:singlePolyRec} is $\ip^*_+$ when $p$ has a zero in $\ringint_L^d$, but it is unknown whether \eqref{eqn:singlePolyRec} is $\ip^*_+$ if one only assumes $p$ is intersective, even in the case $L = \mathbb{Q}$.
More generally, one could ask whether a version of Theorem~\ref{thm:singlePolyRec} holds for a given intersective polynomial $p$ over an arbitrary integral domain $R$.
Under the additional assumption that $p$ has zero constant term it was shown in \cite{MR2145566} that $\{ u \in R : \mu(B \cap T^{p(u)}B) > 0 \}$ has positive density with respect to some F\o{}lner sequence in $R$, but whether this set is syndetic is unknown.
We cannot proceed as in the proof of Theorem~\ref{thm:singlePolyRec}, or apply \cite[Theorem~1.8]{MR1417769}, at such a level of generality due to complications that arise when the additive group of the ring is not finitely generated.
However, if the ring is a countable field then we have proved in \cite{arxiv:1409.6774} the following version of Theorem~\ref{thm:intComb}.

\begin{theorem}
Let $W$ be a finite-dimensional vector space over a countable field $F$ and let $T$ be an action of the additive group of $W$ on a probability space $(X,\mathscr{B},\mu)$.
For any polynomial mapping $\phi : F^n \to W$ with $\phi(0) = 0$, any $B \in \mathscr{B}$ and any $\epsilon > 0$ the set
\begin{equation}
\label{eqn:fieldLargeRec}
\{ u \in F^n : \mu(B \cap T^{\phi(u)} B) > \mu(B)^2 - \epsilon \}
\end{equation}
is $\aip^*$ in $F^n$.
\end{theorem}

Actually, it is shown that \eqref{eqn:fieldLargeRec} has the stronger property of being $\aip^*_r$.
See \cite{arxiv:1409.6774} for the details.

The rest of the paper runs as follows.
In the next section we discuss some preliminary results from ergodic theory necessary for proving our results.
Theorem~\ref{thm:singlePolyRec} is proved in Section~\ref{sec:singlePolyRec}.
In Section~\ref{sec:gowersNorms} we recall the definition of Gowers-Host-Kra seminorms for actions of $\mathbb{Z}^m$ and show in Section~\ref{sec:characteristicFactors} that, in our setting, they control the averages \eqref{eqn:recAverage}.
The proof of Theorem~\ref{thm:mainTheorem} is given in Section~\ref{sec:multipleRecurrence}.

\section{Preliminaries}
\label{sec:preliminaries}

In this section we recall some relevant facts about notions of largeness in countable abelian groups and about idempotent ultrafilters that we will need in order to prove our main result.
We also give a version of the well-known ergodic decomposition of $T \times T$ for an ergodic action $T$ of $\mathbb{Z}^m$.
Recall that a subset $S$ of an abelian group $G$ is \define{syndetic} if there is a finite set $F$ such that $S - F = G$.

\begin{lemma}
\label{lem:allFolnerSyndetic}
Let $G$ be a countable abelian group and let $S \subset G$.
Then $S$ is syndetic if and only if $\upperdens_\Phi(S) > 0$ for every F\o{}lner sequence $\Phi$ in $G$.
\end{lemma}
\begin{proof}
First suppose $S$ is not syndetic.
Fix a F\o{}lner sequence $\Psi$ in $G$.
Since $S$ is not syndetic we can find for each $N \in \mathbb{N}$ some $h_N$ in $G$ such that $(\Psi_N + h_N) \cap S = \emptyset$.
With $\Phi_N = \Psi_N + h_N$ we have $\upperdens_\Phi(S) = 0$.

On the other hand, if $S$ is syndetic then $S - F = G$ for some finite, non-empty subset $F$ of $G$ so for any F\o{}lner sequence $\Phi$ we have
\begin{equation*}
1 = \frac{|G \cap \Phi_N|}{|\Phi_N|} \le \sum_{x \in F} \frac{|(S - x) \cap \Phi_N|}{|\Phi_N|}
\end{equation*}
for every $N \in \mathbb{N}$ and therefore $\upperdens_\Phi(S) \ge 1/|F|$.
\end{proof}

This lets us prove that all $\aip^*_+$ sets are syndetic.
As we will see in Example~\ref{ex:aipNotSyndetic}, there are syndetic sets that are not $\aip^*_+$.

\begin{lemma}
\label{lem:sumptuousSyndetic}
Let $G$ be a countable, abelian group.
Then every $\aip^*_+$ subset of $G$ is syndetic.
\end{lemma}
\begin{proof}
Every $\ip^*$ subset of $G$ is syndetic, for if $S \subset G$ is not syndetic then for every finite subset $F$ of $G$ we have $S - F \ne G$.
This allows us to inductively construct an IP set in $G \setminus S$.
Indeed, assuming that we have found $x_1,\dots,x_n \in G \setminus S$ such that
\begin{equation*}
\fs(x_1,\dots,x_n) := \bigg\{ \sum_{n \in \alpha} x_n : \emptyset \ne \alpha \subset \{1,\dots,n\} \bigg\}
\end{equation*}
is disjoint from $S$, choose $x_{n+1}$ outwith $S - \fs(0,x_1,\dots,x_n)$.

Let $A \subset G$ be $\ip^*_+$ and let $B \subset G$ have zero upper Banach density.
Shifts of syndetic sets are themselves syndetic so $A$ is syndetic by the above argument, and therefore has positive upper density with respect to every F\o{}lner sequence.
Now $\upperdens_\Phi(B) = 0$ for every F\o{}lner sequence, so $\upperdens_\Phi(A \setminus B) > 0$ for every F\o{}lner sequence.
It now follows from Lemma~\ref{lem:allFolnerSyndetic} that $A \setminus B$ is syndetic.
\end{proof}

We will also need the following result, which states that if the average of a non-negative sequence is positive along every F\o{}lner sequence, then the averages along F\o{}lner sequences are uniformly bounded away from zero.

\begin{lemma}
\label{lem:positiveFolnerConstant}
Let $G$ be a countable abelian group.
If $\phi : G \to [0,\infty)$ has the property that
\begin{equation}
\label{eqn:positiveFolnerProof}
\liminf_{N \to \infty} \frac{1}{|\Phi_N|} \sum_{u \in \Phi_N} \phi(u) > 0
\end{equation}
for every F\o{}lner sequence $\Phi$ in $G$, then there is some $c > 0$ such that
\begin{equation*}
\liminf_{N \to \infty} \frac{1}{|\Phi_N|} \sum_{u \in \Phi_N} \phi(u) \ge c
\end{equation*}
for every F\o{}lner sequence $\Phi$ in $G$.
\end{lemma}
\begin{proof}
If not then for every $k \in \mathbb{N}$ there is a F\o{}lner sequence $\Phi_k$ such that
\begin{equation*}
0 \le \liminf_{N \to \infty} \frac{1}{|\Phi_{k,N}|} \sum_{u \in \Phi_{k,N}} \phi(u) < \frac{1}{k}
\end{equation*}
and defining $\Phi_N = \Phi_{k_N,N}$ with $k_N \to \infty$ sufficiently quickly gives a F\o{}lner sequence $\Phi$ for which \eqref{eqn:positiveFolnerProof} does not hold.
\end{proof}

\begin{lemma}
\label{lem:positiveFolnerConstantSyndetic}
Let $G$ be a countable amenable group.
If $\phi : G \to [0,\infty)$ is bounded and \eqref{eqn:positiveFolnerProof} holds for every F\o{}lner sequence then there is a constant $c > 0$ such that $\{ u \in G : \phi(u) \ge c \}$ is syndetic.
\end{lemma}
\begin{proof}
Choose $c$ as in the conclusion of Lemma~\ref{lem:positiveFolnerConstant}.
We claim that $A = \{ u \in G : \phi(u) \ge c/2 \}$ is syndetic.
If not then $\upperdens_\Phi(A) = 0$ for some F\o{}lner sequence $\Phi$ by Lemma~\ref{lem:allFolnerSyndetic}.
But
\[
c \le \limsup_{N \to \infty} \frac{1}{|\Phi_N|} \sum_{u \in \Phi_N} \phi(u) 1_A(u) + \limsup_{N \to \infty} \frac{1}{|\Phi_N|} \sum_{u \in \Phi_N} \phi(u) 1_{X \setminus A}(u) \le c/2
\]
makes this impossible.
\end{proof}

\begin{lemma}
\label{lem:finiteIndexDensity}
Let $G$ be a countable abelian group and let $H \subset G$ be a finite index subgroup.
Then
\begin{equation*}
\lim_{N \to \infty} \frac{|H \cap \Phi_N|}{|\Phi_N|} = \frac{1}{[G:H]}
\end{equation*}
for all F\o{}lner sequences $\Phi$ in $G$.
\end{lemma}
\begin{proof}
Let $g_1,\dots,g_k$ be coset representatives for $H$.
We have
\begin{equation*}
\lim_{N \to \infty} \frac{|H \cap \Phi_N|}{|\Phi_N|} - \frac{|(g+H) \cap \Phi_N|}{|\Phi_N|} = 0
\end{equation*}
for any $g \in G$ so
\begin{equation*}
1
=
\limsup_{N \to \infty} \frac{|(g_1 + H) \cap \Phi_N|}{|\Phi_N|} + \cdots + \frac{|(g_k + H) \cap \Phi_N|}{|\Phi_N|}
=
k \limsup_{N \to \infty} \frac{|H \cap \Phi_N|}{|\Phi_N|}
\end{equation*}
with the same holding for the limit inferior.
\end{proof}

Given a F\o{}lner sequence $\Phi$ in a countable abelian group $G$ and a sequence $g \mapsto \phi(g)$ from $G$ to a normed vector space $(X,\nbar\cdot\nbar)$, write
\begin{equation*}
\clim_{g \to \Phi} \phi(g) = x \Leftrightarrow \lim_{N \to \infty} \frac{1}{|\Phi_N|} \sum_{g \in \Phi_N} \phi(g) = x
\end{equation*}
and
\begin{equation*}
\dlim_{g \to \Phi} \phi(g) = x \Leftrightarrow \lim_{N \to \infty} \frac{1}{|\Phi_N|} \sum_{g \in \Phi_N} \nbar\phi(g) - x\nbar = 0.
\end{equation*}
If $\dlim_{g \to \Phi} \phi(g) = x$ we say that $\phi(g)$ converges along $\Phi$ \define{in density} to $x$.
The following lemma is immediate.

\begin{lemma}
\label{lem:dlimChebyshev}
Let $g \mapsto \phi(g)$ be a sequence from a countable abelian group $G$ to a normed vector space $(X,\nbar\cdot\nbar)$ and let $\Phi$ be a F\o{}lner sequence in $G$.
If
\begin{equation*}
\dlim_{g \to \Phi} \phi(g) = x
\end{equation*}
then $\upperdens_\Phi(\{ g \in G : \nbar \phi(g) - x \nbar \ge \epsilon \}) = 0$ for every $\epsilon > 0$.
\end{lemma}

Variations of the van der Corput trick play a role in most polynomial ergodic theorems.
We will make use of the following version.

\begin{proposition}
\label{prop:hilbertVdc}
Let $G$ be an abelian group and $\mathscr{H}$ be a Hilbert space over $\mathbb{C}$.
Let $g : G \to \mathscr{H}$ be a bounded map.
Then
\begin{equation*}
\limsup_{N \to \infty} \bnbar \frac{1}{|\Phi_N|} \sum_{u \in \Phi_N} g(u) \bnbar^2 \le \frac{1}{|\Phi_H|} \sum_{h \in \Phi_H} \limsup_{N \to \infty} \frac{1}{|\Phi_N|} \sum_{u \in \Phi_N} \langle g(u+h), g(u) \rangle
\end{equation*}
for any F\o{}lner sequence $\Phi$ in $G$ and any $H$ in $\mathbb{N}$.
\end{proposition}
\begin{proof}
\cite[Lemma~4]{MR2151605}.
\end{proof}

Recall that an \define{ultrafilter} on a non-empty set $X$ can be defined as a filter that is maximal with respect to containment.
We will make use of the following characterization of distal systems in terms of limits along idempotent ultrafilters.
This characterization is briefly described below.
For more details, see \cite{MR2052273} and \cite{MR2893605}.

\begin{definition}
Given an ultrafilter $\ultra{p}$ on a group $G$, a map $\phi$ from $G$ to a topological space $X$ and a point $x \in X$, write
\begin{equation}
\label{eqn:ultrafilterLimit}
\lim_{g \to \ultra{p}} \phi(g) = x
\end{equation}
if $\{ g \in G : \phi(g) \in U \} \in \ultra{p}$ for all neighborhoods $U$ of $x$.
\end{definition}

When $X$ is compact and Hausdorff, for any $\phi : G \to X$ there is a unique $x \in X$ such that \eqref{eqn:ultrafilterLimit} holds.

Given a group $G$, one can define an associative binary operation on the set $\beta G$ of ultrafilters on a group $G$ by
\begin{equation*}
\ultra{p} \conv \ultra{q} = \{ A \subset G : \{ g : Ag^{-1} \in \ultra{p} \} \in \ultra{q} \}
\end{equation*}
for all ultrafilters $\ultra{p},\ultra{q}$ on $G$.
An ultrafilter $\ultra{p}$ on $G$ is \define{idempotent} if $\ultra{p} \conv \ultra{p} = \ultra{p}$.
It follows from an application of Ellis's lemma (see \cite[Lemma~1]{MR0101283}) that every semigroup has idempotent ultrafilters.

Let $(X,\mathsf{d})$ be a compact metric space and let $T$ be an action of a group $G$ on $(X,\mathsf{d})$.
Points $x,y \in X$ are said to be \define{proximal} if
\begin{equation*}
\inf \{ \mathsf{d}(T^g x, T^g y) : g \in G \} = 0
\end{equation*}
and the action is \define{distal} if no two distinct points are proximal.
As the next lemma shows, for distal systems limits along idempotent ultrafilters are always the identity.

\begin{lemma}
\label{lem:distalIp}
Let $G$ be a group and let $T$ be a distal action of $G$ on a compact metric space $(X,\mathsf{d})$ by continuous maps.
Then
\begin{equation}
\label{eqn:distalIp}
\lim_{g \to \ultra{p}} T^g x = x
\end{equation}
for every $x \in X$ and every idempotent ultrafilter $\ultra{p}$ on $G$.
\end{lemma}
\begin{proof}
Fix $x \in X$ and an idempotent ultrafilter $\ultra{p}$ in $\beta G$.
We have
\begin{equation*}
\lim_{g \to \ultra{p}} T^g \big( \lim_{h \to \ultra{p}} T^h x \big) = \lim_{g \to \ultra{p}} \lim_{h \to \ultra{p}} T^{gh} x = \lim_{g \to \ultra{p}} T^g x =: y
\end{equation*}
because $\ultra{p} \conv \ultra{p} = \ultra{p}$ so $x$ and $y$ are proximal.
By distality they must be equal.
\end{proof}

\begin{corollary}
\label{cor:distalIpRec}
Let $G$ be a group and let $T$ be a distal action of $G$ on a compact metric space $(X,\mathsf{d})$.
For every $x \in X$ and every neighborhood $U$ of $x$ the set $\{ g \in G : T^g x \in U \}$ is $\ip^*$.
\end{corollary}
\begin{proof}
Fix $x \in X$ and let $U$ be a neighborhood of $x$.
Since $T$ is distal we have $\{ g \in G : T^g x \in U \} \in \ultra{p}$ for every idempotent ultrafilter $\ultra{p}$ on $G$.
But any set that belongs to every idempotent ultrafilter is $\ip^*$ (see \cite{MR2893605} for details).
\end{proof}

One can use minimal idempotent ultrafilters to exhibit syndetic sets that are not $\aip^*_+$.
Recall that an idempotent ultrafilter $\ultra{p} \in \beta G$ is \define{minimal} if it is minimal with respect to the order $\ultra{p} \le \ultra{q}$ defined by the relation $\ultra{p} \conv \ultra{q} = \ultra{q} \conv \ultra{p} = \ultra{p}$.
A set $S \subset G$ is \define{central} or a \define{$\C$ set} if it belongs to some minimal idempotent ultrafilter, a \define{$\C^*$ set} if its intersection with every $\C$ set is non-empty, and a \define{$\C^*_+$ set} if it is a shift of a $\C^*$ set.

\begin{example}
\label{ex:aipNotSyndetic}
Following the proof of \cite[Theorem~2.20]{MR2052273} one can construct a $\C^*_+$ subset of $\mathbb{Z}^m$ that is not syndetic.
Therefore, in order to produce a syndetic set that is not $\aip^*_+$, it suffices to show that every $\aip^*$ subset of $\mathbb{Z}^m$ is a $\C^*$ set.
Let $S$ be an $\aip^*$ set and write $S = A \setminus B$ where $A$ is $\ip^*$ and $\upperdens(B) = 0$.
Certainly $A$ is $\C^*$.
But every central set has positive upper Banach density by \cite[Theorem~2.4(iii)]{MR2052273}, so $A \setminus B$ remains $\C^*$.
\end{example}

The last result about ultrafilters in this section is about limits along polynomials having zero constant term.
We will use it in the proof of Lemma~\ref{lem:invSubgroupEig}.

\begin{lemma}
\label{lem:ultraPolyLimit}
Let $R$ be a commutative ring and let $G$ be an abelian, compact, Hausdorff topological group.
Fix an additive homomorphism $\psi : R \to G$.
For any $k \in \mathbb{N}$, any polynomial $p \in R[x_1,\dots,x_k]$ with $p(0) = 0$, and any idempotent ultrafilter $\ultra{p}$ on the additive group of $R^k$ we have $\lim\limits_{r \to \ultra{p}} \psi(p(r)) = 0$.
\end{lemma}
\begin{proof}
The proof is by induction on the degree of $p$.
When $p$ has degree 1 the map $r \mapsto \psi(p(r))$ is an additive homomorphism so we have
\begin{equation}
\label{eqn:ultraPolyLimit}
\begin{aligned}
\lim_{r \to \ultra{p}} \psi(p(r))
&
=
\lim_{r \to \ultra{p}} \lim_{s \to \ultra{p}} \psi(p(r + s))
\\
&
=
\lim_{r \to \ultra{p}} \lim_{s \to \ultra{p}} \psi(p(r)) + \psi(p(s))
=
2 \lim_{r \to \ultra{p}} \psi(p(r))
\end{aligned}
\end{equation}
by idempotence so the limit in question is zero.

For the induction step, write $\psi(p(r+s)) = \psi(p(r)) + \psi(p(s)) + \psi(q(r,s))$ for some polynomial $q$ with twice as many indeterminates as $p$ and zero constant.
By induction we have
\[
\lim_{r \to \ultra{p}} \lim_{s \to \ultra{p}} \psi(q(r,s)) = 0
\]
so we again have \eqref{eqn:ultraPolyLimit} and the limit in question is zero.
\end{proof}

We conclude this section with the following well-known result about the ergodic decomposition of $T \times T$ when $T$ is an ergodic action of $\mathbb{Z}^m$ on a compact metric probability space $(X,\mathscr{B},\mu)$.
By a \define{$\mathbb{Z}^m$-system} we mean a tuple $\mathbf{X} = (X,\mathscr{B},\mu,T)$ where $(X,\mathscr{B},\mu)$ is a compact metric probability space and $T$ is an action of $\mathbb{Z}^m$ on $(X,\mathscr{B},\mu)$ by measurable, measure-preserving transformations.

Recall that the \define{Kronecker factor} of an ergodic system $(X,\mathscr{B},\mu,T)$ is the factor corresponding to the closed subspace of $\lp^2(X,\mathscr{B},\mu)$ spanned by the eigenfunctions of $T$.
Since $T$ is ergodic \cite[Theorem~1]{MR0172961} implies that the Kronecker factor $(Z,\mathscr{Z},\haar,T)$ has the structure of a compact abelian group equipped with Haar measure on which $T$ corresponds to a rotation determined by a homomorphism $\mathbb{Z}^m \to Z$ with dense image.

\begin{theorem}
\label{thm:prodErgDecomp}
Let $\mathbf{X} = (X,\mathscr{B},\mu,T)$ be an ergodic $\mathbb{Z}^m$ system with Kronecker factor $\mathbf{Z} = (Z,\mathscr{Z},\haar,T)$.
For each $s$ in $Z$ define a measure $\mu_s$ on $(X \times X,\mathscr{B} \otimes \mathscr{B})$ by
\begin{equation*}
\int f_1 \otimes f_2 \intd\mu_s = \int \condex{f_1}{\mathbf{Z}}(z) \cdot \condex{f_2}{\mathbf{Z}}(z-s) \intd\haar(z)
\end{equation*}
for all $f_1,f_2$ in $\lp^\infty(\mathbf{X})$.
Then $\mu_s$ is the ergodic decomposition of $\mu \otimes \mu$.
\end{theorem}
\begin{proof}
The Kronecker factor $(X,\mathscr{Z},\haar)$ has the structure of a compact abelian group.
Let $\alpha : \mathbb{Z}^m \to Z$ be a homomorphism with dense image that determines $T$ on $(Z,\mathscr{Z},\haar)$.
Write $\pi$ for the factor map $\mathbf{X} \to \mathbf{Z}$.

Write $\mathbf{X} \times \mathbf{X}$ for the system $(X^2,\mathscr{B} \otimes \mathscr{B}, \mu \otimes \mu,T \times T)$.
If $F$ in $\lp^2(\mathbf{X} \times \mathbf{X})$ is invariant then $F$ is $\pi^{-1}\mathscr{Z} \otimes \pi^{-1}\mathscr{Z}$ measurable.
This is because any $T \times T$-invariant function can be approximated by linear combinations of products of eigenfunctions of $T$.
It follows that $F$ is of the form $\Psi \circ \pi$ for some $\Psi$ in $\lp^2(\mathbf{Z} \times \mathbf{Z})$.
Thus we can write $\Psi$ as
\begin{equation*}
\Psi = \sum_{i,j} c_{i,j}  \chi_i \otimes \chi_j
\end{equation*}
where $\chi_i$ is an orthonormal basis of $\lp^2(\mathbf{Z})$ consisting of characters.
Invariance of $\Psi$ gives
\begin{equation}
\label{eqn:KroneckerProductDecomp}
\Psi = (T \times T)^n \Psi = \sum_{i,j} c_{i,j} \chi_i(n \cdot \alpha) \chi_j(n \cdot \alpha) \chi_i \otimes \chi_j
\end{equation}
for all $n$ in $\mathbb{Z}^d$.
Thus $c_{i,j}(1 - \chi_i(n \cdot \alpha) \chi_j(n \cdot \alpha)) = 0$ for all $n$ in $\mathbb{Z}^d$ and all $i,j$.
If $c_{i,j}$ is non-zero for some $i,j$ we have $\chi_i(n \cdot \alpha) \chi_j(n \cdot \alpha) = 1$ for all $n$ in $\mathbb{Z}$, and the character $\chi_i \chi_j$ takes the value 1 on the orbit of $\alpha$ so it is constant.
Thus if $c_{i,j}$ is non-zero we have $\chi_i = \overline{\chi}_j$, leading to the simplification
\begin{equation}
\label{eqn:KroneckerProductDecompReduced}
\Psi = \sum_{i} c_i \cdot \chi_i \otimes \overline{\chi}_i
\end{equation}
of \eqref{eqn:KroneckerProductDecomp}.
For any $i$ and any subset $U$ of $\mathbb{C}$ we have
\begin{equation*}
\begin{aligned}
(\chi_i \otimes \overline{\chi}_i)^{-1}U &= \{ (z_1,z_2) : \chi_i(z_1 -z_2) \in U \} = \{ (z_1,z_2) : z_1-z_2 \in \chi_i^{-1} U \}
\end{aligned}
\end{equation*}
so $\chi_i\pi \otimes \overline{\chi}_i\pi$ is measurable with respect to the sub-$\sigma$-algebra
\begin{equation*}
\mathscr{I} = \sigma ( \{ (x_1,x_2) : \pi x_1 - \pi x_2 \in A \} : A \in \mathscr{Z} )
\end{equation*}
of $\mathscr{B} \otimes \mathscr{B}$.
Since $F$ was an arbitrary invariant function in $\lp^2(\mathbf{X} \times \mathbf{X})$ and every set in $\mathscr{I}$ is invariant under $T \times T$, we have that $\mathscr{I}$ is the sub-$\sigma$-algebra of $T \times T$-invariant sets.

This suggests that for each $s \in Z$ there is a measure on
\begin{equation*}
\{ (x_1,x_2) : \pi x_1 - \pi x_2 = s \}
\end{equation*}
that is ergodic for $T \times T$.
To make this precise, fix $s \in Z$ and let $\haar_s$ be the measure on $Z^2$ obtained by pushing $\haar$ forward using the map $z \mapsto (z, z - s)$.
Then, let $\mu_s$ be the measure on $(X^2,\mathscr{B}^2)$ defined by
\begin{equation*}
\int f_1 \otimes f_2 \intd\mu_s = \int \condex{f_1}{\mathbf{Z}} \otimes \condex{f_2}{\mathbf{Z}} \intd\haar_s
\end{equation*}
for all $f_1,f_2$ in $\lp^\infty(X,\mathscr{B},\mu)$.
By definition of $\mu_s$ we have
\begin{equation*}
\int f_1 \otimes f_2 \intd\mu_s = \int \condex{f_1}{\mathbf{Z}}(z) \cdot \condex{f_2}{\mathbf{Z}}(z-s) \intd\haar(z)
\end{equation*}
for all $f_1,f_2$ in $\lp^\infty(X,\mathscr{B},\mu)$.
This proves $\mu_s$ depends measurably on $s$.
It is immediate that each of the measures $\mu_s$ is $T \times T$-invariant.
Moreover, our description of $\mathscr{I}$ implies that if $C$ is $T \times T$-invariant then $\mu_s(C)$ must be either 0 or 1, so each of the measures $\mu_s$ is ergodic.
Lastly, note that
\begin{equation*}
\begin{aligned}
\iint f_1 \otimes f_2 \intd\mu_s \intd\haar(s)
&
=
\iint \condex{f_1}{\mathbf{Z}}(z) \cdot \condex{f_2}{\mathbf{Z}}(z-s) \intd\haar(z) \intd\haar(s)
\\
&
=
\int \condex{f_1}{\mathbf{Z}}(z) \int \condex{f_2}{\mathbf{Z}}(z-s) \intd\haar(s) \intd\haar(z)
\\
&
=
\int f_1 \otimes f_2 \intd(\mu \otimes \mu)
\end{aligned}
\end{equation*}
by Fubini's theorem, so $\mu_s$ is the ergodic decomposition of $\mu \otimes \mu$.
\end{proof}

\section{Single Polynomial Recurrence}
\label{sec:singlePolyRec}

In this section we prove Theorem~\ref{thm:singlePolyRec}, which relies on the following lemmas.

\begin{lemma}
\label{lem:groupHomIndex}
Let $L$ be an algebraic number field.
If $p \in \ringint_L[x_1,\dots,x_d]$ and the induced map $\ringint_L^d \to \ringint_L$ is a non-zero homomorphism of abelian groups then $p(\ringint_L^d)$ is a finite-index subgroup of $\ringint_L$.
\end{lemma}
\begin{proof}
Write $p(x_1,\dots,x_d) = a_1 x_1 + \cdots + a_d x_d$.
Certainly the image of $p$ is a subgroup of $\ringint_L$.
Since some $a_i$ is non-zero, $p(\ringint_L^d)$ contains the ideal generated by $a_i$, which is non-zero.
But every non-zero ideal in the ring of integers of an algebraic number field has finite index (see \cite[Section~I.8]{MR1362545}).
\end{proof}

\begin{lemma}
\label{lem:invSubgroupEig}
Let $G$ be an abelian group and let $H$ be a finite index subgroup.
If $T$ is an action of $G$ on a probability space $(X,\mathscr{B},\mu)$ and $f \in \lp^2(X,\mathscr{B},\mu)$ is invariant under $T|H$ then $f$ is a finite sum of eigenfunctions of $T$.
\end{lemma}
\begin{proof}
Let $g_1,\dots,g_n$ be coset representatives for $H$ with $g_1 = 0$.
Writing any $g \in G$ as $h + g_i$ for some $i$ and some $h \in H$, we see that $T^g f = T^{g_i}f$.
Thus the subspace $K$ of $\lp^2(X,\mathscr{B},\mu)$ spanned by $\{ f,\dots,T^{g_n}f \}$ is $T$-invariant.
The unitary representation of $G$ on $K$ decomposes as a direct sum of one-dimensional representations because $G$ is abelian.
In particular $f$ is a sum of eigenfunctions.
\end{proof}

\begin{proof}[Proof of Theorem~\ref{thm:singlePolyRec}]
Let $T$ be an action of the additive group of $\ringint_L$ on a probability space $(X,\mathscr{B},\mu)$.
Fix $B \in \mathscr{B}$ and $\epsilon > 0$.
Let $P$ be the orthogonal projection in $\lp^2(X,\mathscr{B},\mu)$ onto the closed subspace $\mathscr{H}_\mathrm{c}$ spanned by the eigenfunctions of $T$.
Put $f = 1_B - P1_B$.

We begin by proving that
\begin{equation}
\label{eqn:singleDensityLimit}
\clim_{u \to \Phi} | \langle \phi, T^{p(u)} f \rangle|^2 = 0
\end{equation}
for every F\o{}lner sequence $\Phi$ in $\ringint_L$ and every $\phi$ that is orthogonal to $\mathscr{H}_\mathrm{c}$.
Since $\mathscr{H}_\mathrm{c}$ is $T$-invariant we can assume $p(0) = 0$.
First suppose that $p$ has degree 1, so that $p$ is an additive homomorphism $\ringint_L^d \to \ringint_L$.
Lemma~\ref{lem:groupHomIndex} implies $R := p(\ringint_L^d)$ is a finite index subgroup.
Applying the mean ergodic theorem to the product system $(X \times X, \mathscr{B} \otimes \mathscr{B},\mu \otimes \mu, T \times T)$ we see that the limit
\begin{equation*}
\clim_{u \to \Phi} (T \times T)^{p(u)} (f \otimes f)
\end{equation*}
is invariant under $(T \times T)|R$.
By Lemma~\ref{lem:invSubgroupEig} the limit is a sum of eigenfunctions of $T \times T$.
Since the eigenfunctions of $T \times T$ are spanned by functions of the form $\phi_1 \otimes \phi_2$ where $\phi_1$ and $\phi_2$ are eigenfunctions of $T$, we see that \eqref{eqn:singleDensityLimit} is zero when $p$ has degree 1.

The result follows for $p$ of arbitrary degree by applying the van der Corput trick in the product system.
Indeed, given a polynomial $p$ of degree $d$ and putting $g(u) = (T \times T)^{p(u)}(f \otimes f)$, Proposition~\ref{prop:hilbertVdc} gives
\begin{equation*}
\begin{aligned}
&
\limsup_{N \to \infty} \bnbar \frac{1}{|\Phi_N|} \sum_{u \in \Phi_N} (T \times T)^{p(u)} (f \otimes f) \bnbar
\\
\le
&
\frac{1}{|\Phi_1|} \sum_{h \in \Phi_1} \limsup_{N \to \infty} \frac{1}{|\Phi_N|} \sum_{u \in \Phi} \langle f \otimes f, (T \times T)^{p(u)-p(u+h)}(f \otimes f) \rangle = 0
\end{aligned}
\end{equation*}
because, for any fixed $h \in \ringint_L^d$ the polynomial $u \mapsto p(u) - p(u + h)$ has degree smaller than that of $p$.

Since $\Phi$ was an arbitrary F\o{}lner sequence, Lemma~\ref{lem:dlimChebyshev} implies that
\begin{equation*}
\{ u \in \ringint_L^d : |\langle 1_B, T^{p(u)} 1_B \rangle - \langle 1_B, T^{p(u)} P1_B \rangle| \ge \epsilon \}
\end{equation*}
has zero upper Banach density.

Let $f_1,\dots,f_r$ be eigenfunctions of $T$ with eigenvalues $\chi_1,\dots,\chi_r$ such that $\nbar f_1 + \cdots + f_r - P1_B \nbar \le \epsilon$.
Define a map $\psi : \ringint_L^d \to \mathbb{T}^r$ by $\psi(u) = (\chi_1(u),\dots,\chi_r(u))$ for all $u \in \ringint_L^d$.
Let $e_1,\dots,e_m$ be a basis for $\ringint_L$ as a $\mathbb{Z}$-module and write
\begin{equation*}
p(u) = p_1(u) e_1 + \cdots + p_m(u) e_m
\end{equation*}
for polynomials $p_1,\dots,p_m$ in $\mathbb{Z}[x_1,\dots,x_d]$.
We claim that $p_1,\dots,p_k$ are jointly intersective.
Indeed, let $\Lambda = \mathbb{Z} \lambda$ be a finite index subgroup of $\mathbb{Z}$.
Since $p$ is intersective we have $p(\zeta) \in (\lambda e_1 + \cdots + \lambda e_m)$ for some $\zeta$ in $\ringint_L^d$.
This implies $\{ p_1(\zeta),\dots,p_k(\zeta) \} \subset \Lambda$ as desired.
Writing
\begin{equation*}
\psi(p(u)) = p_1(u) (\chi_1(e_1), \dots, \chi_r(e_1)) + \cdots + p_m(u) (\chi_1(e_m), \dots, \chi_r(e_m))
\end{equation*}
we can apply \cite[Proposition~3.6]{MR2435427} to obtain $w$ in $\ringint_L^d$ for which  $|\chi_i(p(w))| < \epsilon/k$ for all $1 \le i \le k$.
The polynomial $q(u) = p(u + w) - p(w)$ has zero constant term.
Thus
\begin{equation*}
\lim_{u \to \ultra{p}} T\psi(q(u)) = 0
\end{equation*}
for any idempotent ultrafilter $\ultra{p}$ on $\ringint_L^d$ by Lemma~\ref{lem:ultraPolyLimit}.
Combining this with how $w$ was chosen, Corollary~\ref{cor:distalIpRec} implies
\begin{equation*}
\begin{aligned}
&
\{ u \in \ringint_L^d  : \langle 1_B, T^{p(u + w)}P 1_B \rangle \ge \mu(B)^2 - \epsilon \}
\\
\supset\,
&
\{ u \in \ringint_L^d : \langle 1_B, T^{p(u + w) - p(w)}P 1_B \rangle \ge \mu(B)^2 - 4\epsilon \}
\end{aligned}
\end{equation*}
is $\ip^*$.
Thus the set
\begin{equation*}
\{ u \in \ringint_L^d  : \langle 1_B, T^{p(u)}P 1_B \rangle \ge \mu(B)^2 - \epsilon \}
\end{equation*}
is $\ip^*_+$ and \eqref{thm:singlePolyRec} is $\aip^*_+$ as desired.
\end{proof}

We now turn to some examples.
Since every non-zero ideal in $\ringint_L$ has finite index, polynomials $p_1,\dots,p_k$ in $\ringint_L[x_1,\dots,x_d]$ are jointly intersective if and only if, for any non-zero ideal $I$ in $\ringint_L$ one can find $\zeta$ in $\ringint_L^d$ such that $\{ p_1(\zeta),\dots,p_k(\zeta) \} \subset I$.
It was shown in \cite[Proposition~6.1]{MR2435427} that when $L = \mathbb{Q}$, polynomials $p_1,\dots,p_k \in \mathbb{Z}[x]$ are jointly intersective if and only if there is an intersective polynomial $p \in \mathbb{Z}[x]$ such that $p | p_i$ for all $1 \le i \le k$.
The same proof works for intersective polynomials of one variable over $\ringint_L$.

\begin{lemma}
Let $L$ be an algebraic number field and let $p_1,\dots,p_k \in \ringint_L[x]$ be jointly intersective.
Then there is an intersective polynomial $p \in \ringint_L[x]$ such that $p|p_i$ for all $1 \le i \le k$.
\end{lemma}
\begin{proof}
Let $p \in \ringint_L[x]$ be the greatest common divisor of $p_1,\dots,p_k$ in $L[x]$.
Then one can find $h_1,\dots,h_k \in L[x]$ such that $h_1 p_1 + \cdots + h_k p_k = p$.
By clearing denominators we obtain $f_1 p_1 + \cdots + f_k h_k = dp$ for polynomials $f_1,\dots,f_k \in \ringint_L[x]$.
Joint intersectivity of $p_1,\dots,p_k$ now implies intersectivity of $dp$ and thus of $p$.
\end{proof}

\begin{example}
Let $K$ be an algebraic number field and fix $c \in \ringint_K$.
Define $f$ in $\ringint_K[x]$ by $f(x) = x ^2 + c$ for all $x \in \ringint_K$.
We show that if $f$ is intersective then $f$ has a root in $\ringint_K$.
The converse is immediate.

Suppose to the contrary that $f$ does not have a root in $\ringint_K$.
Put $L = K(\sqrt{-c})$.
Then $f$ is the minimal polynomial of $\sqrt{-c}$.
Since $f$ is intersective it has a root modulo every prime ideal $\prideal{p}$ in $\ringint_K$.
Thus $f$ is a product of two linear factors in the ring $\ringint_K/\prideal{p}[x]$.
By Kummer's theorem \cite[Page~37]{MR1362545} this implies that $\prideal{p}\ringint_L$ is not prime and therefore factors in $\ringint_L$.
This is a contradiction because one can always find prime ideals in $\ringint_K$ which remain prime when lifted to $\ringint_L$.
Thus $f$ has a root in $\ringint_K$.

For a specific example, consider $f(x) = x^2 + 1$ over $\mathbb{Z}[i]$ and let $T_1$, $T_2$ be commuting, measure-preserving actions of $\mathbb{Z}$ on a probability space $(X,\mathscr{B},\mu)$.
Then $a + ib \mapsto T_1^a T_2^b$ is an action of $\mathbb{Z}[i] = \ringint_{\mathbb{Q}[i]}$ on $(X,\mathscr{B},\mu)$.
Theorem~\ref{thm:singlePolyRec} tells us
\begin{equation*}
\{ u \in \mathbb{Z}[i] : \mu(B \cap T^{p(u)} B) \ge \mu(B)^2 - \epsilon \}
\end{equation*}
is $\aip^*_+$ for any $B \in \mathscr{B}$ and any $\epsilon > 0$.
In terms of $\mathbb{Z}$-actions, we see that
\begin{equation}
\label{eqn:gaussianExample}
\{ (a,b) \in \mathbb{Z}^2 : \mu(B \cap T_1^{a^2 - b^2 + 1} T_2^{2ab} B) \ge \mu(B)^2 - \epsilon \} 
\end{equation}
is $\aip^*_+$ for any $B \in \mathscr{B}$ and any $\epsilon > 0$.

In this case we can actually say more.
By replacing $b$ with $b + 1$ in \eqref{eqn:gaussianExample} we obtain
\[
\{ (a,b) \in \mathbb{Z}^2 : \mu(B \cap T_1^{a^2 - b^2 -2b} T_2^{2ab} B) \ge \mu(B)^2 - \epsilon \} 
\]
and this set is $\ip^*$ by \cite{MR1692634}.
Thus \eqref{eqn:gaussianExample} is $\ip^*_+$.
\end{example}

Note that any non-constant, monic polynomial can be made intersective by passing to an extension in which it has a root.
Our second example is of an intersective polynomial over $\mathbb{Z}[i]$ without a root.
It is based on \cite[Page~3]{MR0195803}.

\begin{example}
Write $L = \mathbb{Q}[i]$ and let $\alpha$ and $\beta$ be primes in $\ringint_L = \mathbb{Z}[i]$ distinct from $1 + i$ such that $\alpha$ is a quadratic residue modulo $(\beta)$ and vice versa.
Assume also that one of $\alpha$, $\beta$ or $\alpha\beta$ is a square modulo $(1 + i)^5$.
Then $f(x) = (x^2 - \alpha)(x^2 - \beta)(x^2 - \alpha\beta)$ in $\ringint_L[x]$ is intersective.

It suffices to prove that $f$ has a root modulo every non-zero ideal in $\ringint_L$.
Since every non-zero, proper ideal in $\ringint_L$ factors a product of powers of prime ideals, the Chinese remainder theorem implies that it suffices to prove $f$ has a root modulo $\mathfrak{p}^n$ for every prime ideal $\mathfrak{p}$ in $\ringint_L$ and every $n \in \mathbb{N}$.

If $\mathfrak{p} = (z)$ for some prime $z$ distinct from $\alpha$, $\beta$ and $1 + i$ then quadratic reciprocity in $\mathbb{Z}[i]$ implies that one of the factors of $f$ has a root modulo $\mathfrak{p}$.
Since the root is non-zero in $\ringint_L/\mathfrak{p}$ Hensel's lemma \cite[Page~105]{MR1362545} implies that the same factor has a root modulo every power of $\mathfrak{p}$.

The same argument shows that $f$ has a root modulo $\mathfrak{p}^n$ when $\mathfrak{p} \in \{ (\alpha), (\beta) \}$ by our assumption that $\alpha$ is a residue modulo $(\beta)$ and vice versa.

Lastly, if $\mathfrak{p} = (1 + i)$ then one of the factors $h$ of $f$ has a root modulo $(1 + i)^n$ for $n \le 5$ by assumption.
Suppose now that $w$ is a root of this factor modulo $(1+i)^5$ for some $n \ge 5$.
Thus $(1 + i)^n$ divides $h(w)$.
If $(1 + i)^{n+1}$ divides $h(w)$ then certainly $h$ has a root modulo $(1 + i)^{n+1}$.
Otherwise $(1+i)^{n+1}$ does not divide $h(w)$ so
\[
h(w + (1+i)^{n-2}) = h(w) - 2 (1+i)^{n-2} h(0) + (1+i)^{2n-4}
\]
is divisible by $(1 + i)^{n+1}$ because $n \ge 5$.
\end{example}

\section{Gowers-Host-Kra norms for commuting actions}
\label{sec:gowersNorms}

In this section we recall the construction of Gowers-Host-Kra seminorms for a $\mathbb{Z}^m$-system $\mathbf{X} = (X,\mathscr{B},\mu,T)$, which is totally analogous to the $m = 1$ case given in \cite{MR2150389}.
See \cite[Section~4.3.6]{griesmerThesis} for more on these seminorms.

One defines inductively a sequence $\mathbf{X}^{[k]}$ of systems as follows.
Put $\mathbf{X}^{[0]} = \mathbf{X}$.
Assuming that $\mathbf{X}^{[k]} = (X^{[k]},\mathscr{B}^{[k]},\mu^{[k]},T_{[k]})$ has been defined, put
\begin{equation*}
X^{[k+1]} = X^{[k]} \times X^{[k]} \qquad \mathscr{B}^{[k+1]} = \mathscr{B}^{[k]} \otimes \mathscr{B}^{[k]} \qquad T_{[k+1]} = T_{[k]} \times T_{[k]}
\end{equation*}
and define $\mu^{[k+1]}$ to be the relatively independent self-joining of $\mu^{[k]}$ over the sub-$\sigma$-algebra $\mathscr{I}_{[k]} \subset \mathscr{B}^{[k]}$ of sets invariant under $T_{[k]}$.
Thus for any $F_0,F_1$ in $\lp^\infty(\mathbf{X}^{[k]})$ we have
\begin{equation*}
\int F_0 \otimes F_1 \intd\mu^{[k+1]}
=
\int \condex{F_0}{\mathscr{I}_{[k]}} \cdot \condex{F_1}{\mathscr{I}_{[k]}} \intd\mu^{[k]}
=
\clim_{n \to \Phi} \int F_0 \cdot T_{[k]}^n F_1 \intd\mu^{[k]}
\end{equation*}
for any F\o{}lner sequence $\Phi$ in $\mathbb{Z}^m$.
For example
\begin{equation*}
\mathbf{X}^{[1]} = (X \times X, \mathscr{B} \otimes \mathscr{B}, T \times T, \mu \otimes_{\mathscr{I}_{[0]}} \mu)
\end{equation*}
where $\mathscr{I}_{[0]}$ is the sub-$\sigma$-algebra of $T$-invariant sets.
In particular $\mu^{[1]} = \mu \otimes \mu$ if $T$ is ergodic.

Given $f$ in $\lp^\infty(\mathbf{X})$ write $f^{[k]}$ for the function
\begin{equation*}
f \otimes \cdots \otimes f = f \circ \pi_1 \cdots f \circ \pi_{2^k}
\end{equation*}
in $\lp^\infty(\mathbf{X}^{[k]})$, where $\pi_1,\dots,\pi_{2^k}$ are the coordinate projections $X^{[k]} \to X$.
For each $k \ge 1$ the $k$th \define{Gowers-Host-Kra seminorm} $\gnbar \cdot \gnbar_k$ on $\lp^\infty(\mathbf{X})$ is defined by
\begin{equation*}
\gnbar f \gnbar_k^{2^k} = \int f^{[k]} \intd\mu^{[k]}
\end{equation*}
for all $f$ in $\lp^\infty(\mathbf{X})$, and $\gnbar f \gnbar_0 = \int f \intd\mu$.
Note that
\[
\gnbar f \gnbar_1^2 = \int f \otimes f \intd\mu^{[1]} =  \int \condex{f}{\mathscr{I}_{[0]}} \cdot \condex{f}{\mathscr{I}_{[0]}} \intd\mu^{[0]}
\]
for all $f$ in $\lp^\infty(\mathbf{X})$ so
\begin{equation}
\label{eqn:ghkInc}
\gnbar f \gnbar_0 \le \gnbar f \gnbar_1
\end{equation}
by Cauchy-Schwarz.
When $k \ge 1$ we have
\begin{equation*}
\gnbar f \gnbar_k^{2^k} = \int \condex{f^{[k-1]}}{\mathscr{I}_{[k-1]}} \cdot \condex{f^{[k-1]}}{\mathscr{I}_{[k-1]}} \intd\mu^{[k-1]}
\end{equation*}
for all $f$ in $\lp^\infty(\mathbf{X})$.
For any $k \ge 0$ and any F\o{}lner sequence $\Phi$ in $\mathbb{Z}^m$ we have
\begin{equation}
\label{eqn:gowInductive}
\begin{aligned}
\clim_{u \to \Phi} \gnbar f \cdot T^u f \gnbar_k^{2^k}
&
=
\clim_{u \to \Phi} \int f^{[k]} \cdot T_{[k]}^u f^{[k]} \intd\mu^{[k]}
\\
&
=
\int \condex{f^{[k]}}{\mathscr{I}_{[k]}} \cdot \condex{f^{[k]}}{\mathscr{I}_{[k]}} \intd\mu^{[k]}
=
\gnbar f \gnbar_{k+1}^{2^{k+1}}
\end{aligned}
\end{equation}
for all $f$ in $\lp^\infty(\mathbf{X})$ by the mean ergodic theorem.

The key feature of the seminorms $\gnbar \cdot \gnbar_k$ is that, for ergodic $\mathbb{Z}^m$-systems their kernels are determined by $T$-invariant sub-$\sigma$-algebras $\mathscr{Z}_k$ of $\mathscr{B}$ that have a strong algebraic structure.
This was proved for $m = 1$ by Host and Kra~\cite{MR2150389} and generalized to arbitrary $m$ by Griesmer as follows.

\begin{theorem}[\cite{griesmerThesis}]
\label{thm:GriesmerFactors}
Let $\mathbf{X} = (X,\mathscr{B},\mu,T)$ be an ergodic $\mathbb{Z}^m$-system.
For each $k \in \mathbb{N}$ there is an invariant sub-$\sigma$-algebra $\mathscr{Z}_k$ of $\mathscr{B}$ with the property that $\gnbar f \gnbar_k = 0$ if and only if $\condex{f}{\mathscr{Z}_k} = 0$.
Moreover, the factor corresponding to $\mathscr{Z}_k$ is an inverse limit of of a sequence of nilrotations of nilpotency degree at most $r$.
\end{theorem}
\begin{proof}
This is a combination of Lemma~4.4.3 and Theorem~4.10.1 in \cite{griesmerThesis}.
\end{proof}

Using Theorem~\ref{thm:prodErgDecomp} we can relate the Gowers-Host-Kra seminorms of an ergodic $\mathbb{Z}^m$-system $(X,\mathscr{B},\mu,T)$ to those of the systems $(X^2,\mathscr{B}^2,T \times T, \mu_s)$ where $\mu_s$ is the ergodic decomposition of $T \times T$.
Write $\mu_s^{[k]}$ for $(\mu_s)^{[k]}$ and $\gnbar \cdot \gnbar_{s,k}$ for the $k$th Gowers-Host-Kra seminorm of the system $(X^2,\mathscr{B}^2,T \times T, \mu_s)$.

\begin{proposition}
\label{prop:gnormForProduct}
Let $T$ be an ergodic, measure-preserving action of $\mathbb{Z}^m$ on a compact metric probability space $(X,\mathscr{B},\mu)$ and let $\mu_s$ be the ergodic decomposition of $T \times T$.
Then
\begin{equation}
\label{eqn:hkErgodicDecomp}
\mu^{[k+1]} = \int \mu_s^{[k]} \intd\haar(s)
\end{equation}
for every $k \ge 0$ and
\begin{equation*}
\gnbar f \gnbar_{k+1}^{2^{k+1}} = \int \gnbar f \otimes f \gnbar_{s,k}^{2^k} \intd\haar(s)
\end{equation*}
for every $f$ in $\lp^\infty(\mathbf{X})$.
\end{proposition}
\begin{proof}
The proof is by induction on $k$.
When $k = 0$ we use ergodicity of $\mu$ and Theorem~\ref{thm:prodErgDecomp} to obtain
\begin{equation*}
\gnbar f \gnbar_1^2 = \int f \otimes f \intd(\mu \otimes \mu) =  \iint f \otimes f \intd\mu_s \intd\haar(s) = \int \gnbar f \otimes f \gnbar_{s,0} \intd\haar(s)
\end{equation*}
for any $f$ in $\lp^\infty(X,\mathscr{B},\mu)$.

Suppose now that \eqref{eqn:hkErgodicDecomp} holds for some $k \ge 0$.
Fix a bounded, measurable function $F : X^{[k+1]} \to \mathbb{R}$.
Write $\Phi_N = \{1,\dots,N\}^m$.
In this proof we will denote the measure with respect to which a conditional expectation is taken using a subscript.

The pointwise ergodic theorem for actions of $\mathbb{Z}^m$ (see \cite[\nopp VIII.6.9]{MR0117523}) tells us that
\begin{equation*}
\lim_{N \to \infty} \frac{1}{|\Phi_N|} \sum_{u \in \Phi_N} T_{[k+1]}^u F = \condex{F}{\mathscr{I}_{[k+1]}}_{\mu^{[k+1]}}
\end{equation*}
almost surely with respect to $\mu^{[k+1]}$.
It also implies that, for $\haar$ almost every $s$, we have
\begin{equation*}
\frac{1}{|\Phi_N|} \sum_{u \in \Phi_N} T_{[k+1]}^u F \to \condex{F}{\mathscr{I}_{[k+1]}}_{\mu_s^{[k]}}
\end{equation*}
almost surely with respect to $\mu^{[k]}_s$.
Thus \eqref{eqn:hkErgodicDecomp} implies that for $\haar$ almost every $s$ we have
\begin{equation*}
\condex{f}{\mathscr{I}_{[k+1]}}_{\mu^{[k+1]}} = \condex{f}{\mathscr{I}_{[k+1]}}_{\mu_s^{[k]}}
\end{equation*}
on a set of full $\mu_s^{[k]}$ measure.
But then
\begin{equation*}
\begin{aligned}
\int F_0 \otimes F_1 \intd\mu^{[k+2]}
&
=
\int \condex{F_0}{\mathscr{I}_{[k+1]}}_{\mu^{[k+1]}} \cdot \condex{F_1}{\mathscr{I}_{[k+1]}}_{\mu^{[k+1]}} \intd\mu^{[k+1]}
\\
&
=
\iint \condex{F_0}{\mathscr{I}_{[k+1]}}_{\mu^{[k+1]}} \cdot \condex{F_1}{\mathscr{I}_{[k+1]}}_{\mu^{[k+1]}} \intd\mu_s^{[k]} \intd\haar(s)
\\
&
=
\iint \condex{F_0}{\mathscr{I}_{[k+1]}}_{\mu_s^{[k]}} \cdot \condex{F_1}{\mathscr{I}_{[k+1]}}_{\mu_s^{[k]}} \intd\mu_s^{[k]} \intd\haar(s)
\\
&
=
\iint F_0 \otimes F_1 \intd\mu_s^{[k+1]} \intd\haar(s)
\end{aligned}
\end{equation*}
for any bounded, measurable functions $F_0,F_1$ on $X^{[k+1]}$ as desired.
\end{proof}

\section{Characteristic factors for some polynomial averages}
\label{sec:characteristicFactors}

In this section we describe characteristic factors for multiparameter correlations of the form
\begin{equation}
\label{eqn:numberFieldPolyAverage}
\int f \cdot T^{p_1(u)} f \cdots T^{p_k(u)} f \intd\mu
\end{equation}
where $T$ is an ergodic action of $\ringint_L$ on a compact metric probability space $(X,\mathscr{B},\mu)$, the function $f$ belongs to $\lp^\infty(X,\mathscr{B},\mu)$ and $p_1,\dots,p_k$ are non-constant polynomials in $\ringint_L[x_1,\dots,x_d]$.
A \define{characteristic factor} for \eqref{eqn:numberFieldPolyAverage} is a $T$ invariant sub-$\sigma$-algebra $\mathscr{C}$ of $\mathscr{B}$ for which
\[
\int f \cdot T^{p_1(u)} f \cdots T^{p_k(u)} f - \condex{f}{\mathscr{C}} \cdot T^{p_1(u)} \condex{f}{\mathscr{C}} \cdots T^{p_k(u)}\condex{f}{\mathscr{C}} \intd\mu \to 0
\]
in $\lp^2(X,\mathscr{B},\mu)$ for every $f \in \lp^\infty(X,\mathscr{B},\mu)$ along some averaging scheme.
We will be concerned with characteristic factors for convergence in density.
Recall that polynomials $p_1,\dots,p_k$ over a ring are said to be \define{essentially distinct} if $p_i - p_j$ is not constant for all $i \ne j$.
Our main goal in this section is the following theorem.

\begin{theorem}
\label{thm:numberFieldFactors}
Let $L$ be an algebraic number field.
Fix polynomials $p_1,\dots,p_k$ in $\ringint_L[x_1,\dots,x_d]$ that are non-constant and essentially distinct.
For any ergodic action $T$ of the additive group of $\ringint_L$ on a compact metric probability space $(X,\mathscr{B},\mu)$ there is $r \in \mathbb{N}$ such that
\begin{equation*}
\dlim_{u \to \Phi} \int f \cdot T^{p_1(u)}f \cdots T^{p_k(u)} f - \condex{f}{\mathscr{Z}_r} \cdot T^{p_1(u)} \condex{f}{\mathscr{Z}_r} \cdots T^{p_k(u)} \condex{f}{\mathscr{Z}_r} \intd\mu = 0
\end{equation*}
for any F\o{}lner sequence $\Phi$ in $\ringint_L$ and any $f_1,\dots,f_k$ in $\lp^\infty(X,\mathscr{B},\mu)$.
\end{theorem}

The remainder of this section constitutes a proof of Theorem~\ref{thm:numberFieldFactors}.
Essentially, we follow Leibman's proof \cite{MR2151605} of convergence of averages of the form \eqref{eqn:numberFieldPolyAverage} for $\mathbb{Z}$-actions to show that the limiting behavior of \eqref{eqn:numberFieldPolyAverage} along any F\o{}lner sequence is controlled by a certain Gowers-Host-Kra seminorm, and then apply Theorem~\ref{thm:GriesmerFactors}.
For this reason we prove only the results that require some modification for our setting.
We then use Proposition~\ref{prop:gnormForProduct} to obtain characteristic factors for $\dlim$ convergence from those obtained for $\clim$ convergence.

We begin with the following lemma.

\begin{lemma}
\label{lem:leibFirstLinear}
Let $p \in \ringint_L[x_1,\dots,x_d]$ be a degree 1 polynomial with zero constant term.
There is a constant $c \ge 0$ such that
\begin{equation}
\label{eqn:leibFirstLinear}
\lim_{N \to \infty} \bnbar \frac{1}{|\Phi_N|} \sum_{u \in \Phi_N} T^{p(u)} f \bnbar \le c \gnbar f \gnbar_2
\end{equation}
for any $f$ in $\lp^\infty(\mathbf{X})$ and any F\o{}lner sequence $\Phi$ in $\ringint_L^d$.
\end{lemma}
\begin{proof}
Write $p(x_1,\dots,x_d) = a_1 x_1 + \cdots + a_d x_d$ for some $a_i$ in $\ringint_L$, not all of which are zero.
By the mean ergodic theorem we have
\begin{equation}
\label{eqn:metForLinear}
\lim_{N \to \infty} \bnbar \frac{1}{|\Phi_N|} \sum_{u \in \Phi_N} T^{p(u)} f \bnbar^2 = \nbar \condex{f}{\mathscr{I}_\mathfrak{a}} \nbar^2
\end{equation}
where $\mathscr{I}_\mathfrak{a}$ is the sub-$\sigma$-algebra of sets invariant under $T^a$ for all $a$ in the ideal $\mathfrak{a}$ generated by $\{ a_1,\dots,a_d \}$.
By Lemma~\ref{lem:groupHomIndex} the ideal $\mathfrak{a}$ is a finite-index subgroup.
Thus
\begin{equation*}
\begin{aligned}
\lim_{N \to \infty} \frac{[\ringint_L:\mathfrak{a}]}{|\Phi_N|} \sum_{u \in \Phi_N} \gnbar f \cdot T^u f \gnbar_1
&
\ge
\lim_{N \to \infty} \frac{1}{|\Phi_N \cap \mathfrak{a}|} \sum_{u \in \Phi_N \cap \mathfrak{a}}  \gnbar f \cdot T^u f \gnbar_1
\\
&
\ge
\lim_{N \to \infty} \frac{1}{|\Phi_N \cap \mathfrak{a}|} \sum_{u \in \Phi_N \cap \mathfrak{a}}  \gnbar f \cdot T^u f \gnbar_0
=
\nbar \condex{f}{\mathscr{I}_\mathfrak{a}} \nbar^2
\end{aligned}
\end{equation*}
for any $f$ in $\lp^\infty(\mathbf{X})$ by Lemma~\ref{lem:finiteIndexDensity}, \eqref{eqn:ghkInc} and the mean ergodic theorem.
Combining the above with \eqref{eqn:metForLinear} and Cauchy-Schwarz gives us
\begin{equation*}
\lim_{N \to \infty} \bnbar \frac{1}{|\Phi_N|} \sum_{u \in \Phi_N} T^{p(u)} f \bnbar^2 \le [\ringint_L:\mathfrak{a}] \left( \lim_{N \to \infty} \frac{1}{|\Phi_N|} \sum_{u \in \Phi_N} \nbar f \cdot T^u f \nbar_1^2 \right)^{1/2}
\end{equation*}
which, upon applying \eqref{eqn:gowInductive}, yields \eqref{eqn:leibFirstLinear} with $c^2 = [\ringint_L:\mathfrak{a}]$.
\end{proof}

\begin{lemma}
\label{lem:leibSecondLinear}
Let $p \in \ringint_L[x_1,\dots,x_d]$ be a degree 1 polynomial with zero constant term.
There is a constant $c \ge 0$ such that
\begin{equation}
\label{eqn:leibSecondLinear}
\lim_{N \to \infty} \frac{1}{|\Phi_N|} \sum_{u \in \Phi_N} \gnbar f \cdot T^{p(u)} f \gnbar_k^{2^k} \le c \gnbar f \gnbar_k^{2^{k+1}}
\end{equation}
for every $f$ in $\lp^\infty(\mathbf{X})$, every F\o{}lner sequence $\Phi$ in $\ringint_L^d$ and every $k$ in $\mathbb{N}$.
\end{lemma}
\begin{proof}
Write $p(x_1,\dots,x_d) = a_1 x_1 + \cdots + a_d x_d$ for some $a_i$ in $\ringint_L$ not all of which are zero, and let $\mathfrak{a}$ be the ideal in $\ringint_L$ generated by $\{a_1,\dots,a_d\}$.
Let $\mathscr{I}_\mathfrak{a}$ be the sub-$\sigma$-algebra of $\mathscr{B}^{[k]}$ consisting of sets that are invariant under $T_{[k]}^a$ for all $a$ in $\mathfrak{a}$.
For any F\o{}lner sequence $\Phi$ in $\ringint_L^d$ and any $f$ in $\lp^\infty(\mathbf{X})$ we have
\begin{equation*}
\begin{aligned}
&
\lim_{N \to \infty} \frac{1}{|\Phi_N|} \sum_{u \in \Phi_N} \gnbar f \cdot T^{p(u)} f \gnbar_k^{2^k}
\\
=
&
\lim_{N \to \infty} \frac{1}{|\Phi_N|} \sum_{u \in \Phi_N} \int f^{[k]} \cdot T_{[k]}^{p(u)} f^{[k]} \intd\mu_{[k]}
\\
=
&
\int \condex{f^{[k]}}{\mathscr{I}_\mathfrak{a}}^2 \intd\mu_{[k]}
\\
\le
&
\lim_{N \to \infty} \frac{[\ringint_L:\mathfrak{a}]}{|\Phi_N|} \sum_{u \in \Phi_N} \gnbar f \cdot T^u f \gnbar_k^{2^k}
=
[\ringint_L:\mathfrak{a}] \gnbar f \gnbar_{k+1}^{2^{k+1}}
\end{aligned}
\end{equation*}
by arguing as in Lemma~\ref{lem:leibFirstLinear}.
\end{proof}

The next step is to obtain a version of Lemma~\ref{lem:leibSecondLinear} for multiple recurrence.

\begin{theorem}
\label{thm:leibMultiLinear}
Let $p_1,\dots,p_k \in \ringint_L[x_1,\dots,x_d]$ be non-constant, essentially distinct linear polynomials with zero constant term.
There is a constant $c \ge 0$ such that
\begin{equation*}
\limsup_{N \to \infty} \bnbar \frac{1}{|\Phi_N|} \sum_{u \in \Phi_N} T^{p_1(u)} f_1 \cdots T^{p_k(u)} f_k \bnbar \le c \gnbar f_1 \gnbar_{k+1} \nbar f_2 \nbar_\infty \cdots \nbar f_k \nbar_\infty
\end{equation*}
for any $f_1,\dots,f_k$ in $\lp^\infty(\mathbf{X})$ and any F\o{}lner sequence $\Phi$ in $\ringint_L$.
\end{theorem}
\begin{proof}
The proof is by induction of $k$.
When $k = 1$ this is just Lemma~\ref{lem:leibFirstLinear}.
Put $g(u) = T^{p_1(u)} f_1 \cdots T^{p_k(u)} f_k$ for each $u$ in $\ringint_L^d$ and note that in $\lp^2(\mathbf{X})$ we have
\begin{equation*}
\begin{aligned}
\langle g(u+h),g(u) \rangle
=
&
\int \prod_{i=1}^k T^{p_i(u)} ( f_i \cdot T^{p_i(h)} f_i ) \intd\mu
\\
=
&
\int f_k \cdot T^{p_k(h)} f_k \prod_{i=1}^{k-1} T^{p_i(u) - p_k(u)} (f_i \cdot T^{p_i(h)} f_i) \intd\mu
\end{aligned}
\end{equation*}
so for any $H$ in $\mathbb{N}$ we have
\begin{equation*}
\begin{aligned}
&
\limsup_{N \to \infty} \bnbar \frac{1}{|\Phi_N|} \sum_{u \in \Phi_N} \prod_{i=1}^k T^{p_i(u)} f_i \bnbar^2
\\
\le
&
\frac{1}{|\Phi_H|} \sum_{h \in H} \nbar f_k \nbar_\infty^2 \, \limsup_{N \to \infty} \bnbar \frac{1}{|\Phi_N|} \sum_{u \in \Phi_N} \prod_{i=1}^{k-1} T^{p_i(u) - p_k(u)} (f_i \cdot T^{p_i(h)} f_i) \bnbar
\\
\le
&
\frac{1}{|\Phi_H|} \sum_{h \in H} C \gnbar f_1 \cdot T^{p_1(h)} f_1 \gnbar_k \nbar f_2 \nbar_\infty^2 \cdots \nbar f_k \nbar_\infty^2
\end{aligned}
\end{equation*}
by the van der Corput inequality and induction.
Applying Cauchy-Schwarz a number of times and then Lemma~\ref{lem:leibSecondLinear} gives the desired result.
\end{proof}

Using a PET induction argument exactly as in \cite{MR2151605}, one can use Theorem~\ref{thm:leibMultiLinear} to obtain the following result, which gives characteristic factors for Ces\`{a}ro averages.

\begin{theorem}
\label{thm:leibPoly}
For any finite collection of non-constant, essentially distinct polynomials $p_1,\dots,p_k$ in $\ringint_L[x_1,\dots,x_d]$ there is $r$ in $\mathbb{N}$ such that for any F\o{}lner sequence $\Phi$ in $\ringint_L$, any action $T$ of $\ringint_L$ on a compact metric probability space $(X,\mathscr{B},\mu)$ and any $f$ in $\lp^\infty(X,\mathscr{B},\mu)$ we have
\begin{equation*}
\clim_{u \to \Phi} \int f \cdot T^{p_1(u)} f \cdots T^{p_k(u)} f - \condex{f}{\mathscr{Z}_r} \cdot T^{p_1(u)}\condex{f}{\mathscr{Z}_r} \cdots T^{p_k(u)} \condex{f}{\mathscr{Z}_r} = 0
\end{equation*}
whenever $\gnbar f \gnbar_r = 0$.
\end{theorem}

The next step is to obtain a version of Theorem~\ref{thm:leibPoly} for $\dlim$ convergence.
To do so we use product systems as in \cite{MR2138068}.
Let $p_1,\dots,p_k$ be non-constant, essentially distinct polynomials in $\ringint_L[x_1,\dots,x_d]$ and let $r \ge 1$ be as in Theorem~\ref{thm:leibPoly}.
Fix an ergodic action $T$ of $\ringint_L$ on a compact metric probability space $(X,\mathscr{B},\mu)$ and let $\mu_s$ be the ergodic decomposition of $\mu \otimes \mu$.
If $f$ in $\lp^\infty(X,\mathscr{B},\mu)$ satisfies $\gnbar f \otimes f \gnbar_{s,r} = 0$ then
\begin{equation}
\label{eqn:productErgodicFactor}
\clim_{u \to \Phi} \int (f \otimes f) \cdot (T \times T)^{p_1(u)} (f \otimes f) \cdots (T \times T)^{p_k(u)} (f \otimes f) \intd\mu_s = 0
\end{equation}
for any F\o{}lner sequence $\Phi$ in $\ringint_L$.
But from Proposition~\ref{prop:gnormForProduct}, if $\gnbar f \gnbar_{r+1} = 0$ then $\gnbar f \otimes f \gnbar_{s,r} = 0$ for almost every $s$, so \eqref{eqn:productErgodicFactor} holds for almost every $s$.
Integrating over $s$ concludes the proof of Theorem~\ref{thm:numberFieldFactors}.


\section{Multiple recurrence for polynomials over rings of integers}
\label{sec:multipleRecurrence}

Let $T$ be an ergodic action of $\ringint_L$ on a compact metric probability space $(X,\mathscr{B},\mu)$.
In the previous section we showed that, by neglecting a set of zero Banach density, it suffices to study the average \eqref{eqn:numberFieldPolyAverage} when $(X,\mathscr{B},\mu)$ is an inverse limit of nilrotations.
The goal of this section is to prove Theorem~\ref{thm:mainTheorem}.
We do so by exhibiting largeness of the set of multiple recurrence times for nilrotations.

\begin{theorem}
\label{thm:nilmanReturns}
Let $L$ be an algebraic number field.
For any jointly intersective polynomials $p_1,\dots,p_k$ in $\ringint_L[x_1,\dots,x_d]$ and any ergodic action $T$ of $\ringint_L$ on a nilmanifold $(G/\Gamma,\haar)$ determined by a homomorphism $a : \ringint_L \to G$, there is $c > 0$ for which the set
\begin{equation}
\label{eqn:nilmanReturns}
\left\{ u \in \ringint_L^d : \int 1_B \cdot T^{p_1(u)} 1_B \cdots T^{p_k(u)} 1_B \intd\haar \ge c \right\}
\end{equation}
is $\aip^*_+$ for every $B \subset G/\Gamma$ with $\haar(B)  > 0$.
\end{theorem}
\begin{proof}
Let $e_1,\dots,e_m$ be a basis for $\ringint_L$ thought of as a $\mathbb{Z}$-module.
Using this basis we can identify $\ringint_L^d$ with $\mathbb{Z}^{dm}$.
For each $1 \le i \le k$ define polynomials $p_{i,1},\dots,p_{i,m} : \mathbb{Z}^{md} \to \mathbb{Z}$ by
\begin{equation*}
p_i(u) = p_{i,1}(u) e_1 + \cdots + p_{i,m}(u) e_m
\end{equation*}
for each $u$ in $\ringint_L^d$.

We claim that the polynomials $\{ p_{i,j} : 1 \le i \le k, 1 \le j \le m \}$ are jointly intersective.
Indeed, fix $\xi$ in $\mathbb{Z} \setminus \{0\}$ and let $\Lambda$ be the ideal in $\ringint_L$ generated by $\lambda e_1 + \cdots + \lambda e_m$.
There is $\zeta$ in $\ringint_L^d$ such that $\{ p_1(\zeta),\dots,p_k(\zeta) \} \subset \Lambda$.
This means that, for each $i$, we can find $t_1,\dots,t_m$ in $\ringint_L$ such that
\begin{equation*}
p_{i,1}(\lambda) e_1 + \cdots + p_{i,m}(\lambda) e_m = (t_1 e_1 + \cdots + t_m e_m)(\lambda e_1 + \cdots + \lambda e_m)
\end{equation*}
from which it follows that $\lambda \divides p_{i,j}(\zeta)$.

Next, we show that \eqref{eqn:nilmanReturns} is syndetic following \cite{MR2435427}.
Fix a nilpotent Lie group $G$ and a closed, cocompact subgroup $\Gamma$.
Let $\haar$ be the $G$-invariant probability measure on the quotient $X := G/\Gamma$.
Fix $B \subset X$ with $\haar(B) > 0$.
Let $a : \ringint_L \to G$ be a group homomorphism and let $T$ be the induced action of $\ringint_L$ on $G/\Gamma$.
Put $a_i = a(e_i)$.
Then
\begin{equation*}
a(p_i(u)) = a(p_{i,1}(u) e_1 + \cdots + p_{i,m}(u) e_m) = a_1^{p_{i,1}(u)} \cdots a_m^{p_{i,m}(u)}
\end{equation*}
for each $1 \le i \le k$ and every $u$ in $\mathbb{Z}^{dm}$.
Define a polynomial sequence $g : \mathbb{Z}^{dm} \to G^{k+1}$ by
\begin{equation*}
g(u) = (1, a_1^{p_{1,1}(u)} \cdots a_m^{p_{1,m}(u)},\dots, a_1^{p_{k,1}(u)} \cdots a_m^{p_{k,m}(u)})
\end{equation*}
for all $u$ in $\mathbb{Z}^{dm}$.
Let $\diagonal$ be the diagonal in $X^{k+1}$ and let $\haar_\diagonal$ be the push-forward of $\haar$ under the embedding of $X$ in $\diagonal$.
By \cite{MR2122920}, the closure
\begin{equation*}
Y = \overline{\bigcup \{ g(u) \diagonal : u \in \ringint_L^d \}}
\end{equation*}
is a finite union of sub-nilmanifolds of $X^{k+1}$ and the sequence $u \mapsto g(u) \haar_\diagonal$ has an asymptotic distribution $\mu$ in its orbit closure that is a convex combination of the Haar measures on the connected components of $Y$.
Thus we have
\begin{equation*}
\begin{aligned}
&
\clim_{u \to \Phi} \int f_0 \cdot T^{p_1(u)} f_1 \cdots T^{p_k(u)} f_k \intd\haar
\\
=
&
\clim_{u \to \Phi} \int f_0 \otimes T^{p_1(u)} f_1 \otimes \cdots \otimes T^{p_k(u)} f_k \intd\haar_\diagonal
\\
=
&
\clim_{u \to \Phi} \int f_0 \otimes f_1 \otimes \cdots \otimes f_k \intd\haar_{g(u)\diagonal}
\\
=
&
\int f_0 \otimes f_1 \otimes \cdots \otimes f_k \intd \mu
\end{aligned}
\end{equation*}
for any continuous functions $f_0,f_1,\dots,f_k : X \to \mathbb{R}$ and any F\o{}lner sequence $\Phi$ in $\mathbb{Z}^{dm}$.
A density argument proves that the same is true for any $f_0,f_1,\dots,f_k$ in $\lp^\infty(X)$.
Thus for any $B$ in $\mathscr{B}$ we have
\begin{equation*}
\clim_{u \to \Phi} \haar(B \cap T^{-p_1(u)} B \cap \cdots \cap T^{-p_k(u)} B) = \mu(B^{k+1})
\end{equation*}
for every F\o{}lner sequence $\Phi$ in $\mathbb{Z}^{dm}$.
Following the argument on Page~376 of \cite{MR2435427} and applying \cite[Proposition~2.4]{MR2435427} yields
\begin{equation*}
\clim_{u \to \Phi} \int 1_B \cdot T^{p_1(u)} 1_B \cdots T^{p_k(u)} 1_B \intd\haar > 0
\end{equation*}
for every F\o{}lner sequence $\Phi$ in $\mathbb{Z}^{dm}$.
By Lemma~\ref{lem:positiveFolnerConstant} there is some $c > 0$ such that
\begin{equation*}
\clim_{u \to \Phi} \int 1_B \cdot T^{p_1(u)} 1_B \cdots T^{p_k(u)} 1_B \intd\haar \ge c
\end{equation*}
for every $\Phi$.
Thus
\begin{equation}
\label{eqn:nilmanReturnsConstant}
\left\{ u \in \ringint_L^d : \int 1_B \cdot T^{p_1(u)} 1_B \cdots T^{p_k(u)} 1_B \intd\haar \ge \frac{c}{2} \right\}
\end{equation}
has positive density with respect to every F\o{}lner sequence and is therefore syndetic by Lemma~\ref{lem:positiveFolnerConstantSyndetic}.

It remains to prove \eqref{eqn:nilmanReturnsConstant} is $\aip^*_+$.
Fix a continuous function $f : X \to [0,1]$ with $\nbar 1_B - f \nbar_1 < c/8(k+1)$.
Define $\varphi : \ringint_L^d \to \mathbb{R}$ by
\begin{equation*}
\varphi(u) = \int f \cdot T^{p_1(u)} f \cdots T^{p_k(u)} f \intd\haar
\end{equation*}
for every $u \in \ringint_L^d$.
By \cite[Theorem~4.3]{ETS:8947387} we can write $\varphi$ as a sum of sequences $\phi + \psi$ where $\phi$ is a nilsequence and
\[
\dlim_{u \to \Phi} \psi(u) = 0
\]
for every F\o{}lner sequence.
Thus there is a nilmanifold $\tilde{X} = \tilde{G}/\tilde{\Gamma}$, a homomorphism $b : \ringint_L^d \to \tilde{G}$, a continuous function $h : \tilde{X} \to \mathbb{R}$ and some $x \in \tilde{X}$ such that $\phi(u) = h(b(u) x)$ for all $u \in \ringint_L^d$.
Combining the above, we obtain
\begin{equation*}
\left| \int 1_B \cdot T^{p_1(u)} 1_B \cdots T^{p_k(u)} 1_B \intd\haar - h(b(u)x) \right| \le \frac{c}{8} + |\psi(u)|
\end{equation*}
for every $u \in \ringint_L^d$.
The set $\{ u \in \ringint_L^d : |\psi(u)| > c/8 \}$ has zero upper Banach density so syndeticity of \eqref{eqn:nilmanReturnsConstant} and Lemma~\ref{lem:allFolnerSyndetic} imply that $h(b(w)x) \ge c/8$ for some $w \in \ringint_L^d$.
The nilrotation $b$ determines is distal by \cite[Theorem~2.2]{MR0219050}, so
\begin{equation}
\label{eqn:ipDistalLimit}
\lim_{v \to \ultra{p}} h(b(v + w)x) = h(b(w)x)
\end{equation}
for every idempotent ultrafilter $\ultra{p}$ in $\beta \ringint_L^d$ by Lemma~\ref{lem:distalIp}.
It follows that
\begin{equation*}
\{ u \in \ringint_L^d : h(b(u)x) \ge c/8 \}
\end{equation*}
is $\ip^*_+$.
Finally, \eqref{eqn:nilmanReturnsConstant} is $\aip^*_+$ as desired.
\end{proof}

In order to deduce Theorem~\ref{thm:mainTheorem} from Theorem~\ref{thm:nilmanReturns} we need the following preliminary result, based on \cite[Proposition~7.1]{MR670131}.

\begin{proposition}
\label{prop:fkoLift}
Fix a countable, commutative ring $R$ and polynomials $p_1,\dots,p_l$ in $R[x_1,\dots,x_d]$.
Let $(X,\mathscr{B},\mu)$ be a compact metric probability space and let $T$ be an action of the additive group of $R$ on $(X,\mathscr{B},\mu)$ by measurable, measure-preserving maps.
Fix $B \in \mathscr{B}$ with $\mu(B) > 0$.
For any countably generated $T$-invariant sub-$\sigma$-algebra $\mathscr{D} \subset \mathscr{B}$ and any $D \in \mathscr{D}$ with $\mu(B \symdiff D) < \mu(B)/8l$ we can find $E \in \mathscr{D}$ with $\mu(E) > 0$ such that
\begin{equation}
\label{eqn:fkoLift}
\int T^{p_1(u)} 1_B \cdots T^{p_l(u)} 1_B \intd\mu \ge \frac{1}{2} \int T^{p_1(u)} 1_E \cdots T^{p_l(u)} 1_E \intd\mu
\end{equation}
for every $u \in R$.
\end{proposition}
\begin{proof}
We have $\mu(D) \ge \mu(B) - \mu(B)/8l > 0$ because $|\mu(B) - \mu(D)| \le \mu(B \symdiff D)$.
Let $x \mapsto \mu_x$ be a disintegration of $\mu$ over $\mathscr{D}$.
Put
\begin{equation*}
E = \{ x \in D : \mu_x(B) > 1 - 1/2l \}
\end{equation*}
and note that
\begin{align*}
\mu(D \setminus B)
&
=
\iint 1_D 1_{X \setminus B} \intd\mu_x \intd\mu(x)
\\
&
=
\int 1_D(x) \mu_x(X \setminus B) \intd\mu(x)
\\
&
\ge
\int 1_{D \setminus E}(x) \left( 1 - \mu_x(B) \right) \intd\mu(x)
\ge
\frac{\mu(D \setminus E)}{2l}
\end{align*}
implies $\mu(D \setminus E) < \mu(B)/4$ as otherwise $\mu(B \symdiff D) < \mu(B)/8l$ is contradicted.
Thus $\mu(E) \ge \mu(B)/2$.
Fix $u \in R$.
If $x \in T^{-p_i(u)}E$ then $\mu_x(T^{-p_1(u)}B) > 1 - 1/2l$ because $\mathscr{D}$ is $T$-invariant.
Thus if $x \in T^{-p_1(u)}E \cap \cdots \cap T^{-p_l(u)}E$ we have
\begin{equation*}
\mu_x(T^{-p_1(u)}B \cap \cdots \cap T^{-p_l(u)}B) > \frac{1}{2}
\end{equation*}
and integrating over $T^{-p_1(u)}E \cap \cdots \cap T^{-p_l(u)}E$ gives \eqref{eqn:fkoLift}.
\end{proof}

Here is the proof of Theorem~\ref{thm:mainTheorem}.

\begin{proof}[Proof of Theorem~\ref{thm:mainTheorem}]
Let $T$ be an ergodic action of $\ringint_L$ on a compact metric probability space $(X,\mathscr{B},\mu)$ and fix $B \in \mathscr{B}$ with $\mu(B) > 0$.
Let $r$ be as in Theorem~\ref{thm:numberFieldFactors}.
Put $h = \condex{1_B}{\mathscr{Z}_r}$.
We can assume that the polynomials $p_1,\dots,p_k$ in $\ringint_L[x_1,\dots,x_d]$ are distinct.
Since distinct, jointly intersective polynomials are always essentially distinct, for every $\epsilon > 0$ the set
\begin{equation*}
\left\{ u \in \ringint_L^d : \left| \int 1_B \cdot T^{p_1(u)} 1_B \cdots T^{p_k(u)} 1_B \intd\mu - \int h \cdot T^{p_1(u)} h \cdots T^{p_k(u)} h \intd\mu \right| \ge \epsilon \right\}
\end{equation*}
has zero upper Banach density by Theorem~\ref{thm:numberFieldFactors}.
Since $h$ is positive on $B$ we can find $C \in \mathscr{B}$ and $a > 0$ such that $a 1_C \le h$.

The factor corresponding to $\mathscr{Z}_r$ is an inverse limit of nilrotations by Theorem~\ref{thm:GriesmerFactors}.
Thus we can find a Borel subset $D$ of a nilrotation such that $\mu(C \symdiff D) \le \mu(C)/8(k+1)$.
Combining Proposition~\ref{prop:fkoLift} with Theorem~\ref{thm:nilmanReturns} implies there is some $c > 0$ such that
\begin{equation*}
\left\{ u \in \ringint_L^d : \int h \cdot T^{p_1(u)} h  \cdots T^{p_k(u)} h \intd\mu \ge c \right\}
\end{equation*}
is $\aip^*_+$.
Picking $\epsilon = c/2$ proves that \eqref{eqn:mainTheorem} is also $\aip^*_+$ as desired.
\end{proof}

We conclude by giving a proof of Theorem~\ref{thm:mainTheoremNonerg}.

\begin{proof}[Proof of Theorem~{\ref{thm:mainTheoremNonerg}}]
Let $T$ be an action of $\ringint_L$ on a compact metric probability space $(X,\mathscr{B},\mu)$ and fix $B \in \mathscr{B}$ with $\mu(B) > 0$.
Let $\mu_x$ be an ergodic decomposition for $\mu$.
For almost every $x$ the set
\begin{equation*}
R_x = \{ u \in \ringint_L^d : \mu_x(B \cap T^{p_1(u)} B \cap \cdots \cap T^{p_k(u)} B) > 0 \}
\end{equation*}
is $\aip^*_+$ by Theorem~\ref{thm:mainTheorem} and therefore syndetic by Lemma~\ref{lem:sumptuousSyndetic}.
Thus for every F\o{}lner sequence $\Phi$ in $\ringint_L^d$ we have
\begin{equation*}
\liminf_{N \to \infty} \frac{1}{|\Phi_N|} \sum_{u  \in \Phi_N} \mu_x(B \cap T^{p_1(u)} B \cap \cdots \cap T^{p_k(u)} B) > 0
\end{equation*}
so Fatou's Lemma implies that
\begin{equation*}
\liminf_{N \to \infty} \frac{1}{|\Phi_N|} \sum_{u \in \Phi_N} \int \mu_x(B \cap T^{p_1(u)} B \cap \cdots \cap T^{p_k(u)} B) \intd\mu > 0
\end{equation*}
and \eqref{eqn:mainTheoremNoErgodic} is syndetic by Lemma~\ref{lem:positiveFolnerConstantSyndetic}.
\end{proof}

\printbibliography

\end{document}